\documentclass[a4paper]{interact}

\usepackage[natbibapa,nodoi]{apacite}
\newtheorem{definition}{Definition}
\newtheorem{theorem}{Theorem}
\newtheorem{lemma}{Lemma}

\newtheorem{example}{Example}
\newtheorem{remark}{Remark}

\usepackage[a4paper,breaklinks,colorlinks]{hyperref}
\renewcommand{\hat}{\widehat}
\renewcommand{\tilde}{\widetilde}
\newcommand{\bR}{\mathbb{R}}
\newcommand{\R}{\mathbb{R}}
\newcommand{\bC}{\mathbb{C}}
\newcommand{\C}{\mathbb{C}}
\newcommand{\st}{^\star}
\newcommand{\bE}{\mathbb{E}}
\newcommand{\trasp}{^\mathsf{T}}
\newcommand{\Rinv}{R^{-1}}
\newcommand{\te}{t_{\rm f}}
\newcommand{\Vdiff}{\mV^{\rm{diff}}_{[E,A,B]}}
\newcommand{\Vsys}{\mV^{\rm{sys}}_{[E,A,B]}}

\newcommand{\expi}{\mathrm{e}}
\DeclareMathOperator{\tr}{tr}
\DeclareMathOperator{\rk}{rank}

\DeclareMathOperator{\re}{Re}

\newcommand{\cI}{\mathcal{I}}

\newcommand{\mV}{\mathcal{V}}
\newcommand{\intl}[4]{\int_{#1}^{#2}{#3}\,\mathrm{d}{#4}}
\newcommand{\dd}{\mathrm{d}}
\newcommand{\Vinit}{L^2_{\mathcal{F}_0}\big(\Vdiff\big)}
\newcommand{\EVinit}{L^2_{\mathcal{F}_0}\big(E\Vdiff\big)}

\usepackage{enumerate}

\title{Discounted Cost Linear Quadratic Gaussian Control for Descriptor Systems}
\author{Hermann Mena, Lena-Maria Pfurtscheller, Matthias Voigt}
\author{
	\name{Hermann Mena\textsuperscript{a}, Lena-Maria 
Pfurtscheller\textsuperscript{b} and Matthias 
Voigt\textsuperscript{c,}\textsuperscript{d}}\thanks{CONTACT M.~Voigt. Email: matthias.voigt@uni-hamburg.de}
	\affil{\textsuperscript{a} Universidad Yachay Tech, Department of 
Mathematics, San Miguel de Urcuqu\'{i}, Ecuador; \textsuperscript{b} 
Universit\"{a}t Innsbruck, Institut f\"{u}r Mathematik, Austria 
\textsuperscript{c}	Universit\"{a}t Hamburg, Fachbereich Mathematik, Germany; 
\textsuperscript{d} Technische Universit\"{a}t Berlin, Institut f\"{u}r 
Mathematik, Germany}}

\begin{document}
	\maketitle
	
	\begin{abstract}	
	
	We consider the linear quadratic Gaussian control problem with a discounted cost functional for descriptor 
systems on the infinite time horizon. Based on recent results from the 
deterministic framework, we characterize the feasibility of this problem using a 
linear matrix inequality. In particular, conditions for existence and uniqueness of optimal controls are derived, which are weaker compared to the standard 
approaches in the literature.  We further show that also for the stochastic 
problem, the optimal control is given in terms of the stabilizing solution of 
the Lur'e equation, which generalizes the algebraic Riccati equation. 
	\end{abstract}
	
	
	\section{Introduction}
	We consider the stochastic system with additive noise
	\begin{equation}
	\dd E x(t) = (Ax(t) + Bu(t))\dd t + N(t)\dd w(t), \quad 	Ex(0) = Ex_0, \label{eq:lqg}\\
	\end{equation}
	where the matrices $E,\,A \in \bR^{n\times n}$ are such that $\det(s 
E-A) \in \bR[s]$ is not the zero polynomial. A matrix pencil $sE-A \in \R[s]^{n \times n}$ fulfilling this 
condition is called \emph{regular}. Let $B \in \bR^{n\times m}$, $N :\bR \to \bR^{n 
\times n_w}$ and $w$ be an $n_w$-dimensional Wiener process. The stochastic process $u:\bR \times \Omega 
\to \bR^m$ is called \emph{input} and $x:\bR \times \Omega \to \bR^n$ \emph{generalized 
state}. We assume that the initial value $Ex_0$ is a random variable that is independent of $w$. 
	We denote the set of such systems~\eqref{eq:lqg} by $\Sigma_{n,m,n_w}$ and 
we write $[E,A,B,N] \in \Sigma_{n,m,n_w}$. Throughout this paper, the 
following notation is used: $(\Omega, \mathcal{A}, (\mathcal{F}_t)_{t\geq 0}, 
\mathbb{P})$ is a complete probability space with sample space $\Omega$, 
$\sigma$-algebra $\mathcal{A}$, filtration $(\mathcal{F}_t)_{t\geq 0}$ and 
probability measure $\mathbb{P}$. We denote by $\bE[\cdot]$ the expectation 
operator with respect to the probability measure $\mathbb{P}$.


For an interval $\cI := [t_1,t_2]$ with $t_1,\,t_2 \in \R \cup \{-\infty,\infty\}$ and $t_1 \le t_2$, we consider stochastic processes $\varphi: \cI \times \Omega 
\to \mathbb{R}^k$. Here we often use the short-hand 
$\varphi(t) := \varphi(t,\cdot)$. In particular, 
$\varphi(\cdot,\omega)$ is interpreted as a particular realization of the 
process while $\varphi(\cdot)$ contains information about its entire 
distribution.
Let 
\begin{multline*}
L^{2,\rm loc}_{\mathcal{F}}(\cI, \bR^k ) :=  \left\{ 
\vphantom{\int_1^2} \varphi : \cI \times \Omega \to \mathbb{R}^k : 
\varphi := (\varphi(t))_{t \in \cI} := (\varphi(t,\cdot))_{t \in \cI} \; 
\text{is} \right. (\mathcal{F}_t)_{t \in \cI}  
 \\ \left. \text{non-anticipating and }  
\int_\mathcal{K} {\bE\left[\Vert \varphi(t) \Vert^2\right]}{\mathrm{d}t} < \infty \text{ for each compact subset } \mathcal{K} \subseteq \cI \right\},
\end{multline*}
where $\| \cdot \|$ denotes the Euclidean norm.
For a subspace $\mathcal{V} \subseteq \bR^k$, we also often make use of the notation 
\begin{multline*}
L^{2,\rm loc}_{\mathcal{F}}(\cI, \mathcal{V} ) := \left\{ \varphi \in L^{2,\rm loc}_{\mathcal{F}}(\cI, \bR^k ) \; : \; \varphi(t,\omega) \in \mathcal{V} \right. \\  \text{for all } t \in \cI \text{ and almost all } \omega \in \Omega \Big\}.
\end{multline*}

Equations of the form~\eqref{eq:lqg} are called stochastic differential 
algebraic equations (SDAEs) or stochastic descriptor systems and the processes 
$(x,u) \in L^{2,\rm loc}_{\mathcal{F}}(\cI, \bR^n ) \times 
L^{2,\rm loc}_{\mathcal{F}}(\cI, \bR^m )$ are called 
\emph{solutions} of $[E,A,B,N]$ on $[0,\te]$, if they solve~\eqref{eq:lqg} on 
the time interval $[0,\te]$. The latter means that
\begin{equation} \label{eq:solution}
 Ex(t,\omega) = Ex_0(\omega) + \int_{0}^t Ax(s,\omega) \mathrm{d}s + \int_{0}^t 
Bu(s,\omega) \mathrm{d}s + \int_{0}^t N \mathrm{d}w(s)  
\end{equation}
is satisfied for all $t \in [0,\te]$ and for almost all $\omega \in \Omega$. In 
\eqref{eq:solution}, the first two integrals have to be understood in the 
Lebesgue sense, while the third integral is an It\^o integral \citep{L74}. 

Note that we need to restrict the set of initial conditions to those that are independent of the noise $w$, so they live in 
\begin{equation*}
 L^2_{\mathcal{F}_0}(\bR^k) := \left\{ x_0 : \Omega \to \R^k \; : \; x_0 \text{ is } \mathcal{F}_0 \text{ non-anticipating and } \bE\big[\|x_0\|^2\big] < \infty \right\},
\end{equation*}
and for a subspace $\mathcal{V} \subseteq \R^k$ we define
\begin{equation*}
 L^2_{\mathcal{F}_0}(\mathcal{V}) := \left\{ x_0 \in L^2_{\mathcal{F}_0}(\bR^k) \, : \, x_0(\omega) \in \mathcal{V} \text{ for almost all } \omega \in \Omega \right\},
\end{equation*}
see for instance~\citet[Section 4.3]{L74}.


	We further consider the \emph{discounted cost functional}
	\begin{align}
	J(x,u,[0,\te]) = \bE\left[ \intl{0}{\te}{ \expi^{2\beta t}\begin{pmatrix}
		x(t) \\ u(t) \end{pmatrix}\trasp \begin{bmatrix} Q & S \\ S\trasp& R \end{bmatrix} 
		\begin{pmatrix}
		x(t) \\ u(t) \end{pmatrix}}{t} \right] \label{eq:ju}
	\end{align}  
	with $\beta < 0$, which is discussed e.\,g., in the works 
by~\citet{BWSV16,BS17} for the standard case $E = I_n$ . The goal is to 
minimize~\eqref{eq:ju} with respect to the state equation~\eqref{eq:lqg}, 
i.\,e., 
we want to minimize~\eqref{eq:ju} over all solutions $(x,u)$ of $[E,A,B,N] \in 
\Sigma_{n,m,n_w}$ on $[0,\te]$. Such problems form a special class of \emph{linear quadratic 
Gaussian (LQG) control problems}. 
	
	The deterministic counterpart of minimizing~\eqref{eq:lqg} with $\beta = 0$ 
with respect to the state equation
	\begin{align}
	\frac{\dd}{\dd t}Ex(t) = Ax(t) + Bu(t), \quad Ex(0) = Ex_0 \in \R^n \label{eq:det}
	\end{align}
	is called the \emph{linear quadratic regulator (LQR) problem} and
	has been widely studied in the literature. The case $E = I_n$ 
    has been analyzed in \cite{Wil71} in which 
    the connection of the LQR problem to a certain linear matrix inequality 
   (the Kalman-Yakubovich-Popov inequality) and quadratic matrix equation 
(the algebraic Riccati equation) has been exposed. A difficulty in 
this context are so-called \emph{singular LQR problems} for which optimal 
control functions do not exist or are not unique for some initial conditions of the ODE. In this case, the algebraic 
Riccati equation does not exist and one must resort to a linear matrix  
equation, called the Lur'e equation, see \cite{Rei11} and the references 
therein. Several attempts to generalize these concepts to differential algebraic equations have been undertaken, see 
e.\,g., \cite{Meh89,Meh91,Gee93,Gee94,KawTK99,KawK02,Mas06} for  
linear-time invariant LQR problems and \cite{KurM04,KunM08,KunM11} for 
time-varying and nonlinear optimal control problems. However, most of the 
analyses in the formerly mentioned works on the linear time-invariant LQR problem suffer from the fact that they 
are only valid if the system under consideration is impulse controllable or if 
the cost functional is nonnegative. A solution for this problem is presented 
in \cite{RRV15,RV18} in which these conditions are not required anymore. This 
analysis is based on a generalization of the Lur'e equation to differential 
algebraic equations instead of an extension of the algebraic Riccati equation. 
This approach turns out be more beneficial even for regular control problems. 
	
	As any real-world system is influenced by uncertainties such as 
environmental influence, stochastic systems are often more 
suitable than purely deterministic ones. If we consider the case that $E$ is 
invertible then the system can be easily transformed to the standard 
stochastic 
linear quadratic regulator problem~\citep{YZ99,RMZ02,SLY16,D04}. For the ODE case with additive 
noise, in~\citet{DGP99,DPD13} have shown that the problem is solvable if the 
matrix $R$ in the cost functional~\eqref{eq:ju} is positive definite and the 
optimal input is given in feedback form which depends on the solution of an 
algebraic Riccati equation. However, for deterministic descriptor systems, the 
problem may also be solvable even if $R$ is not positive definite, 
see~\citet{RV18} for some illustrative examples. 

One of the first works about descriptor systems with additive noise has been 
given 
by~\citet{D89lqg}, where the system is first transferred into 
Weierstra\ss~canonical form and an algorithm for solving the discrete LQG 
problem with noisy output is given. But, by first transferring the system into 
standard form and then solving the problem, one has to assume again that $R>0$ 
to ensure the solvability of the algebraic Riccati equation.~\citet{ZX14} 
considers a descriptor system with multiplicative noise and shows that the 
solution of this problem is given in terms of an algebraic Riccati equation, 
however, still with the positive definiteness restriction in the cost 
functional.~\citet{FCH13} considers the LQG problem in discrete-time with the 
assumption that $S=0$ and further
	\[
	 \begin{bmatrix} Q & 0 \\ 0 & R \end{bmatrix} = \begin{bmatrix} G\trasp \\ D\trasp\end{bmatrix} 
\begin{bmatrix}  G & D\end{bmatrix} 
	\]
	with matrices $G$ and $D$ of appropriate dimensions. In particular, this 
means that $D\trasp G = G\trasp D = 0$. This restrictive assumption has been relaxed 
by~\citet{WL18}, who further considers systems with positive 
semi-definite 
weight matrices $Q = DD\trasp$, $R = GG\trasp$, if 
	\begin{align*}
	\rk\begin{bmatrix}
	0&E&0\\E&A&B
	\end{bmatrix} &= \rk\begin{bmatrix} E &A & N& B	\end{bmatrix} + \rk\, E ,\\
	\rk\begin{bmatrix}
	0&E&0\\E&A&B \\ 0 &G&0\\ 0&0&D	\end{bmatrix} &= n+m + \rk \, E.  	
	\end{align*}
	However, the weight matrix $R$ in the deterministic case can be even 
negative definite~\citep{RV18} while still ensuring existence of a feasible solution of the LQR problem.
	
	Moreover, a common assumption for stochastic descriptor systems is that the 
 uncontrolled system is impulse free, i.\,e., the DAE has at most index one; or 
the corresponding control system is impulse controllable. This means that there exists a feedback control $u(t) = Kx(t)$ such that the closed-loop system has index 
at most one. This assumption has been imposed, e.\,g., by~\citet{GS13}, who has  shown 
that the system is impulse free and mean-square exponentially stable if a 
certain linear matrix inequality is fulfilled. Similarly, the 
works~\citet{ZZS15,ZX14,HM10} deal solely with the impulse free problem. 
Only~\citet{XZ16} has derived conditions for stability and exact observability of 
discrete systems, which do not need to be impulse free, using generalized 
Lyapunov equations. Also~\citet{RV18} have not used the assumption of impulse 
controllability. 

	In this work, we will derive conditions for feasibility, meaning that the 
control problem has a minimal value of the cost functional that will remain finite, 
and regularity, i.\,e., the control problem has a uniquely determined optimal solution trajectory for each consistent initial condition. In this work we will not impose any artificial conditions on the cost functional nor the index of the system just as in \cite{RV18}. We will use 
the dissipation inequality from~\citet{MP09} and extend it to stochastic 
descriptor systems, which in contrast to the works~\citet{ZZLMR16,RH16,XDQ18} 
holds also for the additive noise case. With the help of this inequality we give 
conditions for feasibility of the problem and show its equivalence to the solvability of the Kalman-Yakubovich-Popov inequality.
Moreover, we derive the optimal value of the cost functional 
as a function depending on the maximum solution of this inequality and show that 
also for the stochastic case, the proper extension of the algebraic Riccati 
equation is the Lur'e equation.
	
	
	\section{Formulation of the problem} \label{sec:formulation_nonsingular}
	
	The aim of this section is to give a proper formulation of the problem
    and 
some important definitions used throughout this paper. Moreover, we present the 
results for the particular case of $E$ being nonsingular. 

Assume that $N(t) \equiv N$, i.\,e., a constant matrix. We define the \emph{space of consistent initial differential variables} of 
$[E,A,B,N]\in \Sigma_{n,m,n_w}$ by
	\begin{multline*}
	\mathcal{V}^{\rm diff}_{[E,A,B,N]} := \left\{ x_0(\omega) \in \bR^{n}: \, \exists \, \text{a solution } \, (x,u) \; 
\text{of} \, [E,A,B,N] \right. \\ \left.
          \text{with} \, Ex(0) = Ex_0 \text{ and } (x(\cdot,\omega),u(\cdot,\omega)) \text{ satisfies \eqref{eq:solution}  on $\bR_{\ge 0}$ with } \omega \in \Omega \right\} .
	\end{multline*}
	Note that the space $\mathcal{V}^{\rm diff}_{[E,A,B,N]}$ is not well-defined in general. Conditions for its well-posedness will be devised in Section~\ref{sec:stochastic}.
	
	We further introduce a \emph{value function} 
	\[
	W_+: \EVinit \to \bR \cup\{-\infty,\infty \},
	\]
	which expresses the \emph{optimal cost}, defined by
	\begin{multline} 
	W_+(Ex_0) = \inf\left\{ J(x,u,\bR_{\ge 0}) \; : \; (x,u) \text{ is a 
solution of }  [E,A,B,N]  \text{ on } \bR_{\geq 0} \right. \\ \left.  
\text{with } Ex(0) = Ex_0 \text{ and } \bE[\Vert Ex(\infty)\Vert^2] = 0 \right\} , 
\label{eq:opcost}
	\end{multline}
	where $\bR_{\geq 0}$ denotes the set of nonnegative real numbers. 

	Assume first that $E$ is nonsingular. Then, all 
initial conditions are consistent and the space $\mathcal{V}^{\rm diff}_{[E,A,B,N]}$ is well-defined, i.\,e., $\mathcal{V}^{\rm diff}_{[E,A,B,N]} = \bR^n$ and one can left-multiply 
equation~\eqref{eq:lqg} by $E^{-1}$. If the system is stabilizable and $Q = Q\trasp \geq 0$, $S=0$, and $R = R\trasp > 0$, then the optimal cost for $\te = \infty$ is 
given by~\citep{BS17}
	\begin{align}
	W_+(Ex_0) = \tr\big(P \bE\big[Ex_0(Ex_0)\trasp\big]\big) - \frac{1}{2\beta} \tr\big( PNN\trasp\big), 
\label{eq:cost_invE}
	\end{align}
	where the matrix $P \in \bR^{n\times n}$ is the stabilizing solution of the algebraic Riccati 
equation
	\begin{align}
	0 = E\trasp P A_\beta + 
A_\beta\trasp P E + Q - \big(E\trasp P B )\Rinv \big(E\trasp P B \big)\trasp \label{eq:ARE}
	\end{align}
	with $A_\beta := A + \beta E$. 
	
	Now for our problem with possibly singular $E$ we define 
	\[
	\tilde{x}(t) := \expi^{\beta t} x(t) , \quad \tilde{u}(t) := \expi^{\beta t} 
u(t),
	\]
	and then we have
	\begin{align*}
	\dd E \tilde{x}(t) &=  \beta \expi^{\beta t} Ex(t) \dd t +  \expi^{\beta t} 
\dd Ex(t) \\&=  \beta \expi^{\beta t} Ex(t) \dd t + \expi^{\beta t} (A x(t) + 
Bu(t))\dd t + N\expi^{\beta t}\dd w(t)  
	\\&= \left(A_\beta \tilde{x}(t) + B \tilde{u}(t)\right) \dd t + N\mathrm{e}^{\beta t} \dd 
w(t) . 
	\end{align*}
	This yields the modified problem of minimizing
	\begin{align}
	J(\tilde{x},\tilde{u},\bR_{\ge 0}) = \bE\left[ \intl{0}{\infty}{\begin{pmatrix}
		\tilde{x}(t) \\ \tilde{u}(t) \end{pmatrix}\trasp \begin{bmatrix} Q 
& S \\ S\trasp& R \end{bmatrix} 
		\begin{pmatrix}
		\tilde{x}(t) \\ \tilde{u}(t) \end{pmatrix}}{t} \right] 
\label{eq:ju1}
	\end{align} 
	subject to
	\begin{align}
	\dd E\tilde{x}(t) = (A_\beta \tilde{x}(t) + B \tilde{u}(t))\dd t + N \mathrm{e}^{\beta t} \dd 
w(t) . \label{eq:lqgs}
	\end{align}
	This setting is closer to the deterministic setting, as we get rid of 
the exponential term in the cost functional. Thus, in the following we will 
study the problem~\eqref{eq:ju1}--\eqref{eq:lqgs} and write $x := \tilde{x}, u := 
\tilde{u}, N_{\beta}(t) := N\mathrm{e}^{\beta t}$, and $A := A_\beta$. Note, that the diffusion term in the equation is now time-dependent.
	
	It should be emphasized that the algebraic Riccati equation~\eqref{eq:ARE} 
is independent of $N$ and coincides with the Riccati equation arising in the 
deterministic case, see e.\,g.,~\citet{KS72}. Only the optimal 
value~\eqref{eq:cost_invE} depends on $N$ and coincides with the deterministic 
value if $N = 0$ and the initial condition is deterministic. This motivates to consider the deterministic setting 
in~\citet{V15,RV18} and view the stochastic case as an extension of the 
deterministic problem.~\citet{V15,RV18} have shown that even if the matrices in the 
cost functional are not necessarily positive (semi-)definite, the optimal value 
is still a quadratic function depending on the initial condition, i.\,e., for a consistent initial differentiable variable $Ex_0 \in \R^n$, the 
optimal value is given by
	\[
	x_0\trasp E\trasp P E x_0,
	\] 
	but $P=P\trasp \in\bR^{n\times n}$ solves a Lur'e equation instead of an 
algebraic Riccati equation. In~\citet{V15} this result was deduced for impulse 
controllable systems and~\citet{RV18} extends it to systems which can also be 
non-impulse controllable.

	\section{Stochastic descriptor systems}\label{sec:stochastic}
	In this section we give an overview on important definitions and basic 
concepts for both, stochastic calculus and the theory of descriptor systems. 
	
	Let us consider the uncontrolled stochastic differential equation
	\begin{equation}
	\dd x(t) = Ax(t) \dd t + N\mathrm{e}^{\beta t}\dd w(t) , \quad x(0) = x_0. \label{eq:sde}
	\end{equation}
	Then, the unique solution of this linear equation is given by
	\[
	x(t) = \expi^{At} x_0 + \intl{0}{t}{\expi^{A(t-s)} N\mathrm{e}^{\beta s} }{w(s)}, 
	\]
	where the stochastic integral is meant to be in the sense of 
It\^{o}~\citep{L74}. However, if we consider stochastic descriptor systems, the 
problem becomes more delicate, as their solutions are not always well-defined, as one can 
see in the following example. 
	
	\begin{example}\label{ex:noise}
		Consider
		\[ 
		\dd \begin{bmatrix}	0&1\\0&0 \end{bmatrix} \begin{pmatrix} x_1(t) \\ x_2 (t)	
\end{pmatrix} = \begin{bmatrix}	1&0\\0&1 \end{bmatrix} \begin{pmatrix} x_1(t) \\x_2(t)
\end{pmatrix} \dd t + \begin{bmatrix}
		n_1\\ n_2
		\end{bmatrix} \mathrm{e}^{\beta t}  \dd w(t) .
		\]
		The equations then read as
		\begin{align*}
		\dd x_2(t) &= x_1(t) \dd t + n_1\mathrm{e}^{\beta t}  \dd w(t) ,\\
		0 &=x_2(t) + n_2 \mathrm{e}^{\beta t}  v(t) ,  
		\end{align*}
		where $v$ is the Gaussian white noise process defined by the 
relation~\citep{L74}
		\[
		w(t) = \intl{0}{t}{v(s)}{s}.
		\] 
		Inserting the second equation into the first leads to a derivation of 
the white noise $v$, which is nowhere differentiable and hence, the DAE is not 
well-posed.  
Moreover, the white noise process has to be understood in a 
distributional sense \cite[Chap.~3]{L74}, so $v(t)$ is only used as a symbol. 
Since distributions are excluded in 
our solution concept, we want to avoid the explicit occurrence of the process $v$ 
in our equations. In our example, this means that $n_2=0$ must hold. But then 
the first equation implies
\begin{equation*}
 0 = x_1(t) \mathrm{d}t + n_1 \mathrm{e}^{\beta t}  \dd w(t).
\end{equation*}
Thus, if $n_1 \neq 0$, then $x_1$ would be a white noise process as 
well. So, for well-posedness, we require $\left[\begin{smallmatrix} n_1 \\ n_2
\end{smallmatrix}\right] = 0$. 
\end{example}
	
If no white noise (or its derivatives) appears in the solutions of the SDAE, 
then we call the SDAE well-posed. More precisely, an SDAE $\dd x(t) = Ax(t)\dd t+ N\mathrm{e}^{\beta t}\dd w(t)$ is well-posed, if and only if there exists an initial condition $Ex(0) = Ex_0$ such that the SDAE has a solution on $\bR_{\ge 0}$. The following lemma - which states an equivalent condition for well-posedness - also makes clear, that if an SDAE is well-posed, then it is solvable for all consistent initial conditions.
	\begin{lemma}\label{lemma:noise}
		Let $sE-A \in \mathbb{R}[s]^{n \times n}$ be regular and consider the SDAE
		\begin{align}
		\begin{aligned} 
		\dd E x(t) &= Ax(t) \dd t +  N \mathrm{e}^{\beta t} \dd w(t). \\
		\end{aligned} 
		\end{align}
		Then, the problem is well-posed, if and only if
		\begin{equation} \label{eq:wpcond}
		 \Pi_{\rm r} N = N,
		\end{equation}		
		where $\Pi_{\rm r} =  T^{-1}\left[\begin{smallmatrix} I_{n_1} &0 \\ 0&0 
\end{smallmatrix}\right] T$ is the spectral projector onto the right deflating 
subspace of the matrix pencil $s E-A$ and $T \in \bR^{n\times n}$ is the right 
(nonsingular) transformation matrix for transforming $sE-A$ into  
quasi-Weierstra{\ss} form~\citep{BerIT12}. 
	\end{lemma}
	\begin{proof} The proof is trivial by transforming the system such that the 
pencil $sE-A$ is in quasi-Weierstra{\ss} form. This leads to the decoupled 
system 
	 \begin{align}\label{eq:QWF}
	 \begin{split}
	  \dd x_1(t) &= A_{11} x_1(t) \dd t +N_1 \mathrm{e}^{\beta t} \dd w(t), 
\\
	  \dd E_{22} x_2(t) &= x_2(t) \dd t + N_2 \mathrm{e}^{\beta t} \dd w(t), 
	  \end{split}
	 \end{align}
	 with $A_{11} \in \bR^{n_1 \times n_1}$ and nilpotent $E_{22} \in \bR^{n_2 
\times n_2}$, where $T^{-1} x(t) =: \left[\begin{smallmatrix} x_1(t) \\ x_2(t) 
\end{smallmatrix}\right]$. The first equation in \eqref{eq:QWF} is a standard 
SDE and is uniquely solvable for each choice of $N_1$ and each initial value. The second equation however has the (formal) solution
\begin{equation*}
  x_2(t) = - \sum_{j=0}^{n_2-1} E_{22}^j N_2 (\mathrm{e}^{\beta t} v(t))^{(j)},
\end{equation*}
where $v$ is a Gaussian white 
noise process. Thus, the SDAE is well-posed, if and only if $N_2 = 0$ (and consequently $x_2(0) = 0$ is the only consistent initial condition) which is 
equivalent to condition \eqref{eq:wpcond}.
\end{proof}
	
	As discussed in the previous subsection, one has to ensure that the white 
noise does not occur in the solution of the SDAE.
We will now characterize well-posedness of the spaces of consistent initial 
differential variables and the system space of the stochastic 
descriptor \emph{control} system $[E,A,B,N_{\beta}] \in \Sigma_{n,m,n_w}$ (where again $N_{\beta}(t):=N\mathrm{e}^{\beta t}$). Moreover, we analyze stabilizability of such 
a system. 

First we consider the space of consistent initial differential variables 
$\mathcal{V}^{\rm diff}_{[E,A,B,N]}$ and derive a condition for its well-posedness. For this purpose, we 
make use of the feedback equivalence form given in~\cite{RRV15}.

\begin{definition}
 The system $[E,A,B,N_\beta] \in \Sigma_{n,m,n_\omega}$ is \emph{well-posed}, if 
there exists an initial value $Ex(0) = Ex_0 \in L^2_{\mathcal{F}_0}(\bR^n)$ for which the system has a solution 
$(x,u)$ on $\bR_{\ge0}$. 
\end{definition}

\begin{lemma}[well-posedness of stochastic descriptor control systems] \label{lem:vdiff}
	Let $[E,A,B,N_\beta]\in \Sigma_{n,m,n_\omega}$ with $N_\beta(t) = N\mathrm{e}^{\beta t}$ be given. Then 
there exist invertible matrices $W,\,T \in \bR^{n\times n}$ and $F \in 
\bR^{m\times n}$ such that
	\begin{align}
		W\begin{bmatrix} sE-A & B & N \end{bmatrix} \begin{bmatrix}	T & 0 & 0 
\\ -FT & I_m & 0 \\ 0 & 0 & I_{n_\omega}
		\end{bmatrix} = \begin{bmatrix} sI_{n_1}-A_{11} & 0 & 0 
& B_1 & N_1 \\0&-I_{n_2} & sE_{23} & B_2 & N_2 \\ 0&0&sE_{33}-I_{n_3} & 0 & N_3	
	\end{bmatrix} \label{eq:feedback_transf}
	\end{align}
	where $E_{33} \in \bR^{n_3\times n_3}$ is nilpotent. Then the following 
statements are satisfied:
\begin{enumerate}[a)]
 \item The system $[E,A,B,N_\beta]$ with feedback equivalence form 
\eqref{eq:feedback_transf} is well-posed, if and only if $N_2 = 0$ and $N_3 = 
0$.
 \item If condition a) is satisfied, then the space of consistent initial 
differential variables $\mathcal{V}^{\rm diff}_{[E,A,B,N_\beta]}$ is well-defined and given by 
        \begin{align}
		\mathcal{V}^{\rm diff}_{[E,A,B,,N_\beta]} = T\left( \bR^{n_1+n_2} \times \ker\begin{bmatrix} E_{23} 
\\E_{33} \end{bmatrix}\right). \label{eq:vdiff}
		\end{align}
\end{enumerate}
\end{lemma}

\begin{proof}
	The feedback equivalence form~\eqref{eq:feedback_transf} is an immediate 
consequence 
of~\cite{RRV15}. To show statement a), by Lemma~\ref{lemma:noise} we see that 
the equation obtained by the third row of \eqref{eq:feedback_transf} is 
well-posed, if and only if $N_3 = 0$. If the state is partitioned according to 
the block structure of \eqref{eq:feedback_transf}, i.\,e., $x(t) = 
T \left[\begin{smallmatrix} x_1(t) \\ x_2(t) \\ x_3(t) 
\end{smallmatrix}\right]$, then we directly see that $x_3 \equiv 0$. But then, 
from the second block row of \eqref{eq:feedback_transf}, we get
\begin{equation*}
 0 =  x_2(t) \mathrm{d}t + B_2 u(t) \dd t + N_2\mathrm{e}^{\beta t}  \mathrm{d}w(t) \quad 
\Rightarrow \quad x_2(t) = -B_2 u(t) - N_2\mathrm{e}^{\beta t}  v(t).
\end{equation*}
In other words, well-posedness of this equation is equivalent to $N_2 = 0$.

Statement b) now follows immediately from statement a), since the structure 
is as in the deterministic setting in \cite{RRV15} except that there is a 
stochastic noise term in the differential part of the equation which does not 
affect the set of consistent initial differential variables.
\end{proof}
\begin{remark}
 \begin{enumerate}[a)]
 \item From Lemma~\ref{lem:vdiff} we see that the space $\mathcal{V}^{\rm diff}_{[E,A,B,N_\beta]}$ does not depend on $N_\beta$ in case of a well-posed system. Thus, from now on we will only write $\Vdiff$ instead of $\mathcal{V}^{\rm diff}_{[E,A,B,N_\beta]}$ to simplify notation.
 \item Lemma~\ref{lem:vdiff} makes clear, that if the system $[E,A,B,N_\beta] \in \Sigma_{n,m,n_w}$ has a solution for one initial condition $Ex(0) = Ex_0$, then it has a solution for all $Ex(0) = Ex_0 \in \EVinit$.
 \end{enumerate}
\end{remark}

Similarly as in the deterministic setting, we will impose a stabilizability 
condition in order to guarantee finiteness of the cost functional. Here we use an adaption of the concept of behavioral stabilizability to stochastic DAE from~\cite{BR13}. 
\begin{definition} 
	  A well-posed system $[E,A,B,N_\beta] \in \Sigma_{n,m,n_w}$ with $N_\beta(t) = N\mathrm{e}^{\beta t}$ is called \emph{mean-square stabilizable (in the behavioral sense)} if
	  	for all solutions $(x,u)$ of $[E,A,B,N_\beta]$ there exists a solution $(\tilde{x},\tilde{u}) $ of $[E,A,B,N_\beta]$ with $(\tilde{x}(t,\omega), \tilde{u}(t,\omega)) = (x(t,\omega),u(t,\omega))$ for almost all $\omega \in \Omega$ and $t < 0$ and it holds that 
	  	\[\lim\limits_{t \to \infty} \bE\left[ \left\Vert \begin{pmatrix}\tilde{x}(t) \\ \tilde{u}(t) \end{pmatrix}\right\Vert^2\right] = 0.\]
	\end{definition}
The following characterization holds. 
	\begin{lemma}
		Let $[E,A,B,N_\beta] \in \Sigma_{n,m,n_w}$ with $N_\beta(t) = N\mathrm{e}^{\beta t}$ be well-posed for $t \in \bR_{\ge 0}$ and $\beta < 0$. Then the system~\eqref{eq:lqg} is mean-square stabilizable, if and only if
				\begin{align}
				\rk\begin{bmatrix} \lambda E-A & B	\end{bmatrix} = n \quad \forall 
		\lambda \in \overline{\mathbb{C}^+},  \label{eq:rank_condition}
				\end{align}
			 	where $\overline{\mathbb{C}^+} := \{ \lambda \in 
		\mathbb{C} \; : \; \re(\lambda) \geq 0 \}$ is the closed right complex half-plane. 
	\end{lemma}
\begin{proof}
 
By using the transformation to feedback equivalence form of $[E,A,B,N_\beta]$ and due to the well-posedness of the system we can find invertible matrices $W,\,T \in \R^{n \times n}$ and $F \in \R^{m \times n}$ such that we obtain the transformed system \eqref{eq:feedback_transf} with $N_2 = 0$ and $N_3 = 0$.
Then $(x,u)$ is a solution of $[E,A,B,N_\beta]$, if and only if $\left( \left( \begin{smallmatrix} \tilde{x}_1 \\ \tilde{x}_2 \\ \tilde{x}_3 \end{smallmatrix}\right), \tilde{u}\right) := \left(T^{-1}x , Fx+u\right)$ solves the transformed system; and we have  
\begin{align*}
 \rk \begin{bmatrix} \lambda E -A & B \end{bmatrix} = \rk \begin{bmatrix} \lambda I_{n_1} -A_{11} & B_1 \end{bmatrix} + n_2 + n_3  \quad \forall\,\lambda \in \bC.
\end{align*}

Let \eqref{eq:rank_condition} be satisfied. Then $\rk \begin{bmatrix} \lambda I_{n_1} -A_{11} & B_1 \end{bmatrix} = n_1$ for all $\lambda \in \overline{\bC^+}$. So we can find an $F_1 \in \R^{n_1 \times m}$ such that $\sigma(A_{11}+B_1F_1) \subset \bC^-$. With $\hat{A}_{11} := A_{11} + B_1F_1$, consider the SDE
\begin{equation*}
 \dd \hat{x}_1(t) = \hat{A}_{11} \hat{x}_1(t) \dd t + N_1\mathrm{e}^{\beta t}  \dd w(t)
\end{equation*}
which has the solution
\begin{equation*}
   \hat{x}_1(t) = \expi^{\hat{A}_{11}t} \hat{x}_1(0) + \intl{0}{t}{\expi^{\hat{A}_{11} 
	(t-s)}N_1e^{\beta s} }{w(s)}.
\end{equation*}
In particular, it holds that
	\begin{align*}
	\bE\big[\hat{x}_{1}(t) \hat{x}_{1}(t)\trasp\big] &= \expi^{\hat{A}_{11} t} 
	\bE\big[\hat{x}_{1}(0) \hat{x}_{1}(0)\trasp\big] 
	\expi^{\hat{A}_{11}\trasp  t} + \intl{0}{t} {\expi^{2\beta s}\expi^{\hat{A}_{11}(t-s)}
	N_{1}N_{1}\trasp \expi^{\hat{A}_{11}\trasp(t-s)}}{s} \\
			&= \expi^{\hat{A}_{11} t} \bE\big[\hat{x}_{1}(0)\hat{x}_{1}(0)\trasp\big] \expi^{\hat{A}_{11}\trasp 
	t}+ X(t),  
	\end{align*}
	where $X(\cdot)$ solves the differential Lyapunov equation
	\[
			\dot{X}(t)  = \hat{A}_{11} X(t) + X(t) \hat{A}_{11} \trasp + \expi^{2\beta 
	t}N_{1}N_{1}\trasp, \quad X(0) = 0,
	\]
	see~\cite{BWSV16}. In particular, $X(t) \to 0$ for $t \to \infty$, if $\sigma\big(\hat{A}_{11} \big) \subset \C^-$. Moreover, as $\bE\big[\hat{x}_{1}(t)\trasp \hat{x}_{1}(t)\big] = \tr\left( \bE\big[\hat{x}_{1}(t)\hat{x}_{1}(t)\trasp\big] \right)$ it follows from the positive semi-definiteness of $\bE\big[\hat{x}_{1}(t) \hat{x}_{1}(t)\trasp]$ that $\bE\big[\hat{x}_{1}(t)\trasp \hat{x}_{1}(t)]  = 0$, if and only if $\bE\big[\hat{x}_{1}(t) \hat{x}_{1}(t)\trasp] = 0$. 
	As $\sigma(\hat{A}_{11}) \subset \C^-$, with the feedback $\tilde{u}(t) := F_1\tilde{x}_1(t)$ we obtain $\lim_{t \to \infty} \bE\left[ \left\| \left(\begin{smallmatrix}\tilde{x}_1(t) \\ \tilde{u}(t)\end{smallmatrix}\right)\right\|^2 \right] = 0$ and consequently, $\lim_{t \to \infty} \bE\big[ \big\| \tilde{x}_2(t)\big\|^2 \big] = 0$ as well as  $\lim_{t \to \infty} \bE\big[ \big\|\tilde{x}_3(t) \big\|^2 \big] = 0$. So the system $[E,A,B,N_\beta]$ is mean-square stabilizable.
	
	Conversely assume that $ \rk \begin{bmatrix} \lambda E -A & B \end{bmatrix}< n$ for some $\lambda \in \overline{\bC^+}$. Let $V \in \R^{n_1 \times n_1}$ be an orthogonal matrix leading to the Kalman decomposition
	\begin{equation*}
	  V\trasp A_{11} V = \begin{bmatrix} A_{11}^{(11)} &  A_{11}^{(12)} \\ 0 &  A_{11}^{(22)}  \end{bmatrix}, \quad V\trasp B_1 = \begin{bmatrix} B_1^{(1)} \\ 0 \end{bmatrix}, \quad V\trasp N_1 = \begin{bmatrix} N_1^{(1)} \\ N_1^{(2)} \end{bmatrix}
	\end{equation*}
	with $\sigma\big( A_{11}^{(22)} \big) \cap \overline{\bC^+} \neq \emptyset$. But then by the considerations above, the input-independent SDE
	\begin{equation*}
	\dd \tilde{x}_{1}^{(2)}(t) = A_{11}^{(22)} \tilde{x}_{1}^{(2)}(t) \dd t + N_1^{(2)}\mathrm{e}^{\beta t}  \dd w(t)
	\end{equation*}
	is not mean-square stable, so $[E,A,B,N_\beta]$ is not mean-square stabilizable.
	\end{proof}  

	A crucial role in the study of descriptor systems plays the~\emph{system 
space}.
	\begin{definition}
		The \emph{system space} of $[E,A,B,N_\beta]$ is the smallest subspace 
$\mathcal{V}^{\rm sys}_{[E,A,B,N_\beta]}\subseteq  \bR^{n+m}$ such that 
	\begin{multline*}
		\forall \text{ solutions } (x,u) \text{ of } [E,A,B,N_\beta]: \begin{pmatrix} x(t,\omega) \\ 
u(t,\omega)\end{pmatrix} \in \mathcal{V}^{\rm sys}_{[E,A,B,N_\beta]} \\ \text{ for all } t \in \bR \text{ and all } \omega \in \Omega \text{ for which } (x(\cdot,\omega),u(\cdot,\omega)) \text{ satisfies } \eqref{eq:solution}.
	\end{multline*}
	\end{definition}
	Since the noise only affects the dynamic part of our SDAE, it is easily 
verified that the space $\mathcal{V}^{\rm sys}_{[E,A,B,N_\beta]}$ is equal to system space for deterministic systems as in \citet{RRV15} and hence, it does not depend on $N_\beta$. Therefore, for a more concise notation, we will only write $\Vsys$ instead of $\mathcal{V}^{\rm sys}_{[E,A,B,N_\beta]}$.
	
	\section{The optimal control problem}
	In this section we return to the analysis of the optimal control problem which we state once completely. 
We consider the optimal control problem
	\begin{align} \label{eq:OCP}
	 \begin{split}
      \text{minimize } & J(x,u,\bR_{\ge 0}) = \bE\left[ \int_0^\infty \begin{pmatrix} x(t) \\ u(t) \end{pmatrix}\trasp \begin{bmatrix} Q & S \\ S\trasp & R \end{bmatrix} \begin{pmatrix} x(t) \\ u(t) \end{pmatrix} \dd t \right]	 \\
      \text{subject to } & \dd E x(t) = (Ax(t) + Bu(t)) \dd t + N \mathrm{e}^{\beta t} \dd w(t), \\ & Ex(0) = Ex_0, \quad \lim_{t \to \infty} \bE\left[ \left\| Ex(t) \right\|^2\right] = 0. 
      \end{split}
	\end{align}
	The aim of this section is to derive conditions for feasibility and 
regularity of the corresponding control problem. Recall that we say that $(x,u)$ is a solution of $[E,A,B,N$ on $[0,\te]$ if $(x,u) \in L^{2,\rm loc}_{\mathcal{F}}([0,\te], \bR^n) \times L^{2,\rm loc}_{\mathcal{F}}([0,\te], \bR^m )$ and it solves the stochastic DAE~\eqref{eq:lqg}. The following definition has been adapted from \cite{RV18}.
	
	\begin{definition}[Feasibility, regularity, optimal control]
		Let $[E,A,B,N_\beta] \in \Sigma_{n,m,n_w}$ with $N_\beta(t)=N\mathrm{e}^{\beta t}$ and $\beta < 0$ be well-posed and let $Q = Q\trasp \in 
\bR^{n\times 
n}$, $S \in \bR^{n\times m}$, and $R = R\trasp \in \bR^{m\times m}$ be given.
		\begin{enumerate}[a)]
			\item The optimal control problem~\eqref{eq:OCP} is 
called \emph{feasible} if for all $x_0 \in \Vinit$ it holds that
			\[
			-\infty < W_+(Ex_0) < \infty .
			\]
			\item A solution $(x\st, u\st)$ of $[E,A,B,N_\beta] $ on $\bR_{\geq 0}$ 
with initial condition $Ex(0) = Ex_0 \in \EVinit$ and $\lim_{t \to \infty}\bE\big[\Vert Ex(t)\Vert^2\big] = 0$ 
is called an \emph{optimal control} for~\eqref{eq:OCP}, if
			\[
			W_+(Ex_0) = J(x\st,u\st,\bR_{\ge 0}) .
			\]  
			\item  The optimal control problem~\eqref{eq:OCP} is 
called \emph{regular}, if for all $x_0 \in \Vinit$, there exists a unique 
optimal control. 
		\end{enumerate} 
	\end{definition}

	
	\subsection{Feasibility of the problem} \label{sec:feas}
	The following definition is an extension of dissipativity of stochastic 
differential equations. Our definition is a slight modification from the definitions in~\citet{MP09,ZJ10}.  
	\begin{definition}
		Let $[E,A,B,N_\beta] \in \Sigma_{n,m,n_w}$ with $N_\beta(t)=N\mathrm{e}^{\beta t}$ and $\beta < 0$  be well-posed and let $Q = Q\trasp \in \bR^{n\times 
n}$, $S \in \bR^{n\times m}$, and $R = R\trasp \in \bR^{m\times m}$ be given. Then the system $[E,A,B,N_\beta]$ with the cost functional in~\eqref{eq:OCP} is said to 
be~\emph{dissipative} on $\R_{\ge 0}$ if 
there exist a twice continuously differentiable \emph{storage function} 
$V:E \Vdiff \to \bR$ such that the~\emph{integral dissipation 
inequality}
		\begin{align}
		\bE[V(Ex(t_1))] + \bE\left[\intl{t_1}{t_2}{\mu_N(Ex(t))}{t}\right] \leq  \bE[V(Ex(t_2))] + 
J(x,u,[t_1,t_2])    \label{eq:dissipation}
		\end{align}
		holds for any times $0 \leq t_1 \leq t_2 < \infty$ and any solution $(x,u)$ of $[E,A,B,N_\beta]$ on $[t_1,t_2]$ with
		\begin{equation*}   
		 \mu_N: E \Vdiff \to \bR, \quad Ex_0 \mapsto \frac{1}{2}e^{2\beta t} w(t)\trasp N\trasp \nabla^2V(Ex_0) N w(t).
		\end{equation*}\label{def:dissipative}
	\end{definition}
	Note that the above definition boils down to the well-known dissipation inequality in the deterministic setting with $N = 0$ and deterministic initial condition.

	From now on we also need the following notation: For two symmetric matrices $X,\,Y \in \R^{k \times k}$ and a subspace $\mathcal{V} \subseteq \R^k$ we write $$X =_{\mathcal{V}} (\leq_{\mathcal{V}}, \geq_{\mathcal{V}}) Y \iff v\trasp Xv = 
(\le, \ge) v\trasp Yv \text{ for all } v \in \mathcal{V}. $$  
	
    One can now check that any quadratic function $V : E\Vdiff \to \R$ with $V(Ex_0) = (Ex_0)\trasp P Ex_0$ for which $P \in \R^{n \times n}$ is  a solution of the \emph{Kalman-Yakubovich-Popov (KYP) inequality}
	\begin{align} 
 			\begin{bmatrix}
 			A\trasp P E + E\trasp P A + Q & E\trasp PB+S\\B\trasp PE+S\trasp & R
 			\end{bmatrix} \geq_{\Vsys} 0, \quad P = P\trasp  \label{eq:kyp}
 			\end{align}
	is a storage function in our case. This can be seen together with It\^{o}'s formula. We get
 	\begin{align*}
 	 & \bE[V(Ex(t_2))] - \bE[V(Ex(t_1))] \\&= \bE\left[ \int_{t_1}^{t_2} \dd V(Ex(t)) \right] + \bE\left[\int_{t_1}^{t_2}  \frac{1}{2}e^{2\beta t} w(t)\trasp N\trasp \nabla^2 V(Ex(t)) N w(t)\dd t \right]\\ 
 	 &= \bE\left[ \int_{t_1}^{t_2} (\nabla V(Ex(t)))\trasp \dd Ex(t) \right] + \int_{t_1}^{t_2}  \frac{1}{2}\mathrm{e}^{2\beta t}\tr \big(N\trasp \nabla^2 V(Ex(t)) N\big) \dd t \\ 
 	 &= \bE\left[ \int_{t_1}^{t_2} 2x(t)\trasp E\trasp P (Ax(t) + Bu(t)) \dd t \right] + \int_{t_1}^{t_2} \mathrm{e}^{2\beta t}\tr \big(N\trasp P N\big) \dd t  \\
 	 &\ge -\bE\left[ \int_{t_1}^{t_2} \begin{pmatrix}
 		x(t) \\ u(t) \end{pmatrix}\trasp \begin{bmatrix} Q & S \\ S\trasp& R \end{bmatrix} 
 		\begin{pmatrix}
 		x(t) \\ u(t) \end{pmatrix} \dd t \right] + \bE\left[\int_{t_1}^{t_2} \mu_N(Ex(t))\dd t\right].
    \end{align*}
In the next theorem we will show analogously to the deterministic case that feasibility of the optimal control problem implies solvability of the KYP inequality.
\begin{theorem}
		Let $[E,A,B,N_\beta] \in \Sigma_{n,m,n_w}$ with $N_\beta(t)=N\mathrm{e}^{\beta t}$ and $\beta < 0$  be well-posed and let $Q = Q\trasp \in \bR^{n\times 
		n}$, $S \in \bR^{n\times m}$, and $R = R\trasp \in \bR^{m\times m}$ be given. 
		If the optimal control problem \eqref{eq:OCP} is feasible, then the system $[E,A,B,N_\beta]$ is mean-square stabilizable and the KYP inequality \eqref{eq:kyp} is feasible, i.\,e., there exists at least one matrix $P \in \bR^{n \times n}$ satisfying~\eqref{eq:kyp}.
		\label{thm:wopt}
\end{theorem}

\begin{proof}
	Let $z := \left(\begin{smallmatrix} x \\ u\end{smallmatrix}\right) \in L^{2,\rm loc}_{\mathcal{F}}\big( \bR_{\ge 0},\Vsys\big)$ and $\tilde{z}(t) := z(t) - \bE[z(t)] := z(t)-\hat{z}(t)$. Then $z(t) = \tilde{z}(t) + \hat{z}(t)$ and $\bE[\tilde{z}(t)] = 0$ for all $t \ge 0$. In particular, for $M := \left[ \begin{smallmatrix} Q & S \\ S\trasp & R \end{smallmatrix}\right] \in \bR^{(n+m)\times (n+m)}$ we have
	\begin{align*}
	\bE[z(t)\trasp M z(t)] &= \bE\big[\big(z(t)-\hat{z}(t) + \hat{z}(t)\big)\trasp M \big(z(t)-\hat{z}(t)+ \hat{z}(t)\big)\big] \\
	&= \bE\big[\big(z(t)-\hat{z}(t)\big)\trasp M \big(z(t)-\hat{z}(t)\big)\big] + \bE\big[\hat{z}(t)\trasp M \hat{z}(t)\big] \\ & \qquad + 2 \bE\big[\hat{z}(t)\trasp M \big(z(t)-\hat{z}(t)\big)\big]\\
	&=\bE\big[\tilde{z}(t)\trasp M \tilde{z}(t)\big] + \hat{z}(t)\trasp M \hat{z}(t) + 2 \hat{z}(t)\trasp M \bE\big[\tilde{z}(t)\big] \\
	&= \bE\big[\tilde{z}(t)\trasp M \tilde{z}(t)\big] + \hat{z}(t)\trasp M \hat{z}(t). 
	\end{align*}
	Thus with $\hat{z}(\cdot)=:\left(\begin{smallmatrix} \hat{x}(\cdot) \\ \hat{u}(\cdot) \end{smallmatrix} \right)$ and $\tilde{z}(\cdot)=:\left(\begin{smallmatrix} \tilde{x}(\cdot) \\ \tilde{u}(\cdot) \end{smallmatrix} \right)$ we can write $$J(x,u,\bR_{\ge 0}) = J^{\rm d}\big(\hat{x},\hat{u},\bR_{\ge 0}\big) + J^{\rm s}\big(\tilde{x},\tilde{u},\bR_{\ge 0}\big),$$ where
	\begin{align*}
	J^{\rm d}(\hat{x},\hat{u},\bR_{\ge 0}) &=\int_{0}^{\infty}\begin{pmatrix}\hat{x}(t) \\ \hat{u}(t) \end{pmatrix}\trasp \begin{bmatrix} Q & S\\S\trasp & R \end{bmatrix} \begin{pmatrix}\hat{x}(t) \\ \hat{u}(t) \end{pmatrix} \dd t \\
	J^{\rm s}(\tilde{x},\tilde{u},\bR_{\ge 0})  &= \bE\left[ \int_{0}^{\infty}\begin{pmatrix}\tilde{x}(t) \\ \tilde{u}(t) \end{pmatrix}\trasp \begin{bmatrix} Q & S\\S\trasp & R \end{bmatrix} \begin{pmatrix}\tilde{x}(t) \\ \tilde{u}(t) \end{pmatrix} \dd t \right] 
	\end{align*}
	subject to 
	\begin{align}
	\frac{\dd}{\dd t}E\hat{x}(t) &= A\hat{x}(t) + B\hat{u}(t), \quad E\hat{x}(0) = \bE[Ex_0], \label{eq:ode} \\ 
	\dd E\tilde{x}(t) & = (A\tilde{x}(t) + B\tilde{u}(t)) \dd t + N\mathrm{e}^{\beta t}  \dd w(t) , \quad E\tilde{x}(0) = Ex_0 - \bE[Ex_0]. \label{eq:sden}
	\end{align}	
 Consider the value functions $V_+^{\rm d}:E\Vdiff \to \bR \cup \{-\infty,\infty\}$ and $W_+^{\rm s} : \EVinit \to \bR \cup \{-\infty, \infty\}$ with 
	\begin{align*}
	{V}_+^{\rm d}(Ex_0) &= \inf\left\{ J^{\rm d}(\hat{x},\hat{u},\bR_{\ge 0}): (\hat{x},\hat{u})  \text{ solves } \eqref{eq:ode} \text{ with } \right. \\ & \qquad\qquad\qquad\left. E \hat{x}(0) = \bE[Ex_0] \text{ and } \lim_{t \to \infty} E\hat{x}(t) = 0 \right\}, \\ 
	{W}_+^{\rm s}(Ex_0) &= \inf\Big\{ J^{\rm s}(\tilde{x},\tilde{u},\bR_{\ge 0}): (\tilde{x},\tilde{u})  \text{ solves } \eqref{eq:sden} \text{ with } \\ & \qquad\qquad\qquad \left. E \tilde{x}(0) = Ex_0 - \bE[Ex_0] \text{ and } \lim_{t \to \infty} \bE\big[\left\|E\tilde{x}(t)\right\|^2\big] = 0 \right\}
	\end{align*}
	Since both cost functionals can be minimized independently, we have
	$W_+(Ex_0) = V_+^{\rm d}(Ex_0) + W_+^{\rm s}(Ex_0)$ for each $x_0 \in \Vinit$. Since \eqref{eq:OCP} is feasible and since \eqref{eq:ode} and \eqref{eq:sden} are simultaneously (mean-square) stabilizable in the behavioral sense or not, we must have $V_+^{\rm d}(Ex_0) < \infty$ and $W_+^{\rm s}(Ex_0) < \infty$ for each $x_0 \in \Vinit$. Hence the system $[E,A,B,N_\beta]$ is mean-square stabilizable. Moreover, due to the feasibility of~\eqref{eq:OCP}, we have $V_+^{\rm d}(Ex_0) >- \infty$ for each $x_0 \in \Vinit$. Then according to~\citet[Theorem~3.11]{RV18}, there exists a solution $P \in \bR^{n \times n}$ of the KYP inequality~\eqref{eq:kyp}.

\end{proof}

\begin{remark} If the optimal control problem~\eqref{eq:OCP} is feasible, then due to the stabilizability of \eqref{eq:ode} in the behavioral sense and~\citet[Theorem~3.11]{RV18}, we can even infer the existence of a maximal solution $P_+ \in \bR^{n \times n}$ of~\eqref{eq:kyp}, i.\,e.,
\begin{equation*}
  P_+ \ge_{E\Vdiff} P 
\end{equation*}
for all solutions $P \in \bR^{n\times n}$ of~\eqref{eq:kyp}.
\end{remark}


If $E = I_n$ and the weight matrix $R$ is invertible, then the 
solution of a feasible optimal control problem is given in terms of the solution of an algebraic Riccati equation, see Section~\ref{sec:formulation_nonsingular}. A 
generalization of this algebraic Riccati equation for a possible singular matrix $E$ and a general weight $\left[ \begin{smallmatrix} Q & S \\ S\trasp & R \end{smallmatrix} \right]$ is presented in the following.
	\begin{definition}[{\citet[Definition 2.5]{RV18}}]
		Let $[E,A,B,N_\beta] \in \Sigma_{n,m,n_w}$ with $N_\beta(t)=N\mathrm{e}^{\beta t}$ and $\beta < 0$  be well-posed and let $Q = Q\trasp \in \bR^{n\times 
n}$, $S \in \bR^{n\times m}$, and $R = R\trasp \in \bR^{m\times m}$ be given.
		A triple $(X,K,L) \in \bR^{n\times n} \times \bR^{q \times n} \times 
\bR^{q \times m}$ that fulfills
		\begin{align}
		\begin{bmatrix}
		A\trasp X E + E\trasp XA+Q&E\trasp XB+S\\B\trasp XE+S\trasp & R
		\end{bmatrix} =_{\Vsys} \begin{bmatrix} K\trasp \\L\trasp	
\end{bmatrix}  \begin{bmatrix} K &L	\end{bmatrix} \label{eq:lure} 
		\end{align}
		and 
		\begin{align*}
		\rk_{\bR[s]} \begin{bmatrix}
		-sE+A & B \\ K & L
		\end{bmatrix} = n+q
		\end{align*}
		is called solution of the~\emph{Lur'e equation}, where $\bR[s]$ denotes 
the ring of  polynomials with coefficients in $\bR$. 
		
		A solution $(X,K,L) \in \bR^{n\times n}\times \bR^{q\times n} \times 
\bR^{q\times m} $ of the Lur'e equation is called~\emph{stabilizing}, if 
additionally 
		\begin{align*}
		\rk \begin{bmatrix}
		-\lambda E +A & B \\ K & L
		\end{bmatrix} = n+q  \quad \forall \lambda \in \mathbb{C}^{+} .
		\end{align*}\label{def:lur}
	\end{definition}
Note that if $(P_+,K,L)\in \bR^{n\times n}\times \bR^{q\times n} \times \bR^{q\times m}$ is a stabilizing solution of the Lur'e equation~\eqref{eq:lure}, then $P_+$ is a maximal solution of the KYP inequality~\eqref{eq:kyp}. Conversely, if $P_+ \in \bR^{n \times n}$ is a maximal solution of~\eqref{eq:kyp}, then there exist $K \in \bR^{q \times n}$ and $L \in \bR^{q \times m}$ such that $(P_+,K,L)$ is a stabilizing solution of~\eqref{eq:lure}~\citep{RRV15}.
	
Now we can show the following result which is an adaptation of a result in \citet{RV18}.
	\begin{theorem}
		Let $[E,A,B,N_\beta] \in \Sigma_{n,m,n_w}$ with $N_\beta(t)=N\mathrm{e}^{\beta t}$ and $\beta < 0$  be well-posed and let $Q = Q\trasp \in \bR^{n\times 
n}$, $S \in \bR^{n\times m}$, and $R = R\trasp \in \bR^{m\times m}$ be given. Then the following statements are 
equivalent:
		\begin{enumerate}[a)]
			\item The optimal control problem is feasible, i.\,e., $W_+(Ex_0) 
\in \bR$ for all $x_0 \in \Vinit$. 
			\item The system $[E,A,B,N_\beta]$ is mean-square stabilizable  and the 
KYP 
inequality~\eqref{eq:kyp} has a maximal solution $P_+ = P_+\trasp \in 
\bR^{n\times n}$.
			\item The system $[E,A,B,N_\beta]$ is mean-square stabilizable and there 
exists a storage function $V:E\Vdiff \to \bR$. 
			\item  There exist $q \in \mathbb{N}_0$, $K\in\bR^{q\times n}$, and 
$L \in \bR^{q\times m}$ such that $(P_+,K,L)$ is a stabilizing solution of the 
Lur'e equation~\eqref{eq:lure}. 
		\end{enumerate}
			In the case where the above statements are valid, we have
		\begin{enumerate}[i)]
			\item $W_+(Ex_0) = \tr\big(P_+ \bE\big[Ex_0 (Ex_0)\trasp]\big) - 
\frac{1}{2\beta}\tr\big( P_+ NN\trasp\big) \quad \forall x_0 \in \Vinit$. 
			\item For all $x_0 \in \Vinit$ and solutions $(x,u)$ of 
$[E,A,B,N_\beta]$ on $\bR_{\geq 0}$ with $Ex(0) = Ex_0$ and $\lim\limits_{t \to 
\infty} \bE\big[\Vert Ex(t) \Vert^2\big] = 0 $ it holds that
			\begin{multline}
			J(x,u,\bR_{\ge 0}) = \tr\big(P_+ \bE\big[ E x_0 (Ex_0)\trasp \big]\big) - 
\frac{1}{2\beta}\tr\big( P_+ NN\trasp\big) \\ + \intl{0}{\infty}{\bE\left[ \Vert K x(t) + 
Lu(t) \Vert^2\right]}{t} . \label{eq:jurepres}
			\end{multline}
			\item A solution $(x\st,u\st)$ is an optimal control, if 
and 
only if  $\bE[\Vert Kx\st + Lu\st\Vert^2] = 0$. 
		\end{enumerate} \label{thm:eq_feas_lure}
	\end{theorem}
	\begin{proof}
		Assertion a) $\Rightarrow$ b) follows from Theorem~\ref{thm:wopt}. Statement b) $\Rightarrow$ c) is trivial and c) $\Rightarrow$ a) follows from the dissipation inequality
 since for every solution $(x,u)$ of $[E,A,B,N_\beta]$ with $x_0 \in \Vinit$ and $\lim_{t \to \infty} \bE \big[ \big\| Ex(t)\big\|^2\big] = 0$, we have
 \begin{equation*}
  -\infty < \bE[V(Ex_0)] + \bE\left[\intl{0}{\infty}{\mu_N(Ex(t))}{t}\right] \leq J(x,u,\bR_{\ge 0}).
 \end{equation*}
 The equivalence b) $\Leftrightarrow$ d) follows immediately 
from~\citet[Theorem 3.13]{RV18}. 
		 
		Next we show i)--iii):
		As $P_+ = P_+\trasp\in \bR^{n\times n}$ is the maximal solution of the KYP 
		inequality, it follows from~\citet[Theorem 4.5]{RV18} that there exist a $q \in 
		\mathbb{N}_0$, $K\in \bR^{q\times n}$, and $L \in \bR^{q\times m}$ such that 
		$(P_+,K,L)$ is stabilizing solution of the Lur'e equation~\eqref{eq:lure}. Moreover, from It\^{o}'s formula we obtain 
		\begin{align*}
           \tr\big( P_+ & \bE\big[E x_0 (Ex_0)\trasp \big]\big) - \frac{1}{2\beta}\tr\big( P_+ NN\trasp\big) \\ 
          &= \bE\big[ x_0\trasp E\trasp P_+ E x_0\big] - \frac{1}{2\beta}\tr\big( P_+ NN\trasp\big) \\ 
			&= -\bE\left[\intl{0}{\infty}{}{\left( (Ex(t))\trasp P_+ (Ex(t)) 
		\right)}\right] \\
			&= \bE\left[ \intl{0}{\infty}{\begin{pmatrix}
				x(t) \\ u(t) \end{pmatrix}\trasp \begin{bmatrix} -A\trasp P_+E - 
		E\trasp P_+ A & -E\trasp P_+ B \\ -B\trasp P_+ E& 0 \end{bmatrix} 
			\begin{pmatrix}
				x(t) \\ u(t) \end{pmatrix}}{t} \right] \\
		&= \bE\left[ \intl{0}{\infty}{\begin{pmatrix}
				x(t) \\ u(t) \end{pmatrix}\trasp \left(  \begin{bmatrix} Q  & S \\ 
		S\trasp& R \end{bmatrix} -
				\begin{bmatrix} K\trasp K  & K\trasp L \\ L\trasp K& L\trasp L 
		\end{bmatrix} \right) 
				\begin{pmatrix} 
				x(t) \\ u(t) \end{pmatrix}}{t} \right] \\
			&= J(x,u,\bR_{\ge 0}) - \intl{0}{\infty}{\bE\left[ \Vert Kx(t) + Lu(t) 
		\Vert^2 \right]}{t} , 
		\end{align*}
		which is assertion ii). In particular, a solution $(x\st,u\st)$ of $[E,A,B] $ on $\bR_{\geq 0}$ with  $Ex(0) = Ex_0$ and  $\lim\limits_{t \to \infty} \bE\big[\Vert Ex(t)\Vert^2\big] = 0$ is 
		an optimal control, if and only if $ Kx\st + Lu\st= 0$ almost surely (assertion iii)). 
			The main difference to the deterministic case is the extra term $- \frac{1}{2\beta}\tr\big( P_+ NN\trasp\big)$ due to the function $\mu_N(\cdot)$ in the dissipation inequality~\eqref{eq:dissipation}. Thus we can proceed analogously to~\cite{RV18} by taking the infimum over all possible solutions $(x,u)$ with $Ex(0) = Ex_0$ and $\lim\limits_{t \to \infty} \bE\big[\Vert Ex(t)\Vert^2\big] = 0$. It follows that $W_+(Ex_0) = \tr \big(P_+ \bE\big[E x_0 (Ex_0)\trasp\big]\big)- \frac{1}{2\beta}\tr\big( P_+ NN\trasp\big)$ for all $x_0 \in \Vinit$, since $\intl{0}{\infty}{\bE\left[ \Vert Kx(t) + Lu(t) 
		\Vert^2 \right]}{t}$ can be made arbitrarily small which gives assertion i).
	\end{proof}
	
		\subsection{Regularity of the problem} \label{sec:reg}
	Next, we want to characterize the regularity of the optimal control 
problem. From Theorem~\ref{thm:eq_feas_lure}, we have seen that the solution $(x\st,u\st)$ of $[E,A,B,N_\beta]$ with $Ex(0) = Ex_0$ and $\lim_{t \to \infty}\bE\big[\Vert 
Ex(t)\Vert^2\big] = 0$ is an optimal control if and only if it fulfills the optimality DAE
	\begin{multline}
	\dd\begin{bmatrix}
	E &0\\0&0 \end{bmatrix} \begin{pmatrix}
	x\st(t) \\u\st(t) \end{pmatrix} = \begin{bmatrix} A & B \\ K & L \end{bmatrix} 
\begin{pmatrix}
	x\st(t) \\u\st(t) \end{pmatrix} \dd t + \begin{bmatrix}
	N \\ 0
	\end{bmatrix}\mathrm{e}^{\beta t}  \dd w(t), \\ Ex(0) = Ex_0,  \quad \lim_{t \to \infty}\bE\big[\Vert 
Ex(t)\Vert^2\big] = 0. \label{eq:opt_dae} 
	\end{multline}
	Moreover, as no noise source appears in the algebraic constraints, this DAE is well-posed. Now we can immediately state a regularity result which directly follows from \citet[Theorem~4.7]{RV18}.
	\begin{theorem}
	Let $[E,A,B,N_\beta] \in \Sigma_{n,m,n_w}$ with $N_\beta(t)=N\mathrm{e}^{\beta t}$ and $\beta < 0$  be well-posed and let $Q = Q\trasp \in \bR^{n\times 
n}$, $S \in \bR^{n\times m}$, and $R = R\trasp \in \bR^{m\times m}$ be given.
		Let the optimal control problem~\eqref{eq:OCP} be feasible and $(P,K,L) \in \bR^{n 
\times n} \times \bR^{q\times n}\times \bR^{q\times m}$ be a stabilizing 
solution of the Lur'e equation~\eqref{eq:lure}. Then the following two 
statements are equivalent:
		\begin{enumerate}[a)]
			\item The optimal control problem is regular.
			\item The conditions 
			\begin{equation}
			\ker \begin{bmatrix}
			-\imath\omega E + A & B \\ K & L	\end{bmatrix} = \{0\} \quad \forall\,\omega \in \R, 
\label{cond:pencil} \end{equation}
and
 \begin{multline}
			\begin{bmatrix} E & 0\\ 0&0	\end{bmatrix} \cdot \Vsys + 
\begin{bmatrix} A & B\\ K &L		\end{bmatrix} \cdot \left( \left(\ker E 
\times \bR^{m}\right) \cap \Vsys\right)  \\   = \begin{bmatrix} 
E & 0\\ 0&0	\end{bmatrix} \cdot \Vsys + \begin{bmatrix} A & B\\ K &L		
\end{bmatrix} \cdot \Vsys \label{cond:sys_space}
			\end{multline}
			are fulfilled.
		\end{enumerate}
	\end{theorem}
	
	\begin{proof}
    We see that if $[E,A,B,N_\beta]$ is well-posed, then also the optimality DAE~\eqref{eq:opt_dae} is well-posed. We further know that~\eqref{cond:pencil} 
and~\eqref{cond:sys_space} is equivalent to the unique solvability of the 
deterministic system. Thus, we conclude that there exists a unique solution of the stochastic DAE~\eqref{eq:opt_dae} and the optimal control problem is 
regular, if and only if~\eqref{cond:pencil} and~\eqref{cond:sys_space} are 
fulfilled.   
	\end{proof}
	\begin{remark} If the optimal control problem~\eqref{eq:OCP} is regular, then in the Lur'e equation~\eqref{eq:lure} we have $q=m$. If further $L \in \bR^{m \times m}$ is invertible, then the second row of the optimality DAE~\eqref{eq:opt_dae} can be resolved to obtain the explicit feedback control law $u(t) = -L^{-1}K x(t)$.
	\end{remark}

	\section{Conclusions and Outlook}
	
	In this paper, we have discussed feasibility and regularity of optimal 
control problems for linear stochastic descriptor system with additive noise by 
extending the results of the deterministic case to the stochastic problem. The 
feasibility of the control problem has been characterized by the dissipation 
inequality and the Kalman-Yakubovich-Popov inequality. We have constructed the 
solution of the problem using the stabilizing solution of the Lur'e equation and 
have derived conditions for existence and uniqueness of optimal controls.     
	
	Instead of the discounted cost functional~\eqref{eq:ju}, other functionals 
are often used. In the following we give a short overview of the used cost 
functionals and give an outlook on how to deal with these cost functionals in 
the same manner as with the discounted functional. 
	
	For example, ~\citet{KS72} used the time-average functional
	\[
	J^{\rm a}(x,u,[0,\te]) = \frac{1}{\te} \bE\left[ \intl{0}{\te}{\begin{pmatrix}
		x(t) \\ u(t) \end{pmatrix}\trasp \begin{bmatrix} Q & S \\ S\trasp& 
R \end{bmatrix} 
		\begin{pmatrix}
		x(t) \\ u(t) \end{pmatrix}}{t} \right]
	\]
	and especially, for $\te = \infty$ we obtain 
	\[
	J^{\rm a}(x,u,\bR_{\ge 0}) = \lim\limits_{\te \to \infty}\frac{1}{\te} \bE\left[ 
\intl{0}{\te}{\begin{pmatrix}
		x(t) \\ u(t) \end{pmatrix}\trasp \begin{bmatrix} Q & S \\ S\trasp& 
R \end{bmatrix} 
		\begin{pmatrix}
		x(t) \\ u(t) \end{pmatrix}}{t} \right].
	\] 
	Then, for the standard case with $E = I_n$, $\left[\begin{smallmatrix} Q & S \\ S\trasp & R \end{smallmatrix}\right]\geq 0$, and $R> 0$, the value 
function is given by
	\[
	W_+(x_0) = \tr\big(P^{\rm{a}}NN\trasp\big),
	\]
	where $P^{\rm{a}}$ solves the algebraic Riccati equation 
	\[
	0 = P^{\rm{a}} A + 
A\trasp P^{\rm{a}} + Q - (P^{\rm{a}} B + S)\Rinv (P^{\rm{a}} B + S)\trasp, 
	\]
	which coincides with the Riccati equation arising from the discounted cost 
functional. Thus, one can probably show similar results as presented in this 
paper. Note that in this case, the value function only depends on the matrix $N$ 
and is independent of the initial condition. Moreover, the Lur'e equation and the KYP 
inequality have the same form as for the discounted functional.    
	
	The whole situation changes when considering DAEs with multiplicative noise, 
i.\,e.,
	\[
	\dd Ex(t) = (A x(t) + Bu(t)) \dd t + (Cx(t) + Du(t)) \dd w(t) , \quad Ex(0) 
= Ex_0 ,
	\]
	and with the cost functional 
	\[
			J^{\rm m}(x,u,\bR_{\ge 0}) = \bE^0\left[ \int_{0}^{\infty} 
		\begin{pmatrix}
		x(t) \\ u(t) \end{pmatrix}\trasp \begin{bmatrix} Q & S \\ S\trasp& R 
		\end{bmatrix} 
		\begin{pmatrix}
		x(t) \\ u(t) \end{pmatrix}
		\dd t \right], 
	\]
	where $\bE^0[\cdot]:= \bE[\cdot \,\vert\, \mathcal{F}_0]$ denotes the conditional expectation. For the standard case with (possibly) indefinite weight matrices in the cost functional, the optimal value is
	\[
	W^{\rm{m}}_+(x_0) = x_0\trasp P^{\rm{m}} x_0 ,
	\]
	where $P^{\rm{m}}$ solves the \emph{bilinear} Riccati equation
	\begin{multline*}
	0 = P^{\rm{m}} A + A\trasp P^{\rm{m}} + Q + C\trasp P^{\rm{m}} C \\ - 
(BP^{\rm{m}} + S + D\trasp P^{\rm{m}}C)\trasp (R+D\trasp P^{\rm{m}} D)^{-1} 
(BP^{\rm{m}} + S + D\trasp P^{\rm{m}}C), 
	\end{multline*}
	where $R + D\trasp P^{\rm{m}} D > 0$, see~\cite{CLZ98} for further details. 
	Thus, different linear matrix inequalities and Lur'e equations have to be derived and the theory developed for the deterministic case has to be extended. This is out of the scope of this paper and subject of future considerations.

	\section*{Funding}
	L.-M. Pfurtscheller was supported by a scholarship of the Vizerektorat f\"ur Forschung, University of Innsbruck and by a scholarship of the German Academic Exchange Service (DAAD).

	\bibliographystyle{apacite} 
	\bibliography{biblio}

\begin{thebibliography}{}

\bibitem [\protect \citeauthoryear {%
Arnold%
}{%
Arnold%
}{%
{\protect \APACyear {1974}}%
}]{%
L74}
\APACinsertmetastar {%
L74}%
\begin{APACrefauthors}%
Arnold, L.%
\end{APACrefauthors}%
\unskip\
\newblock
\APACrefYear{1974}.
\newblock
\APACrefbtitle {{S}tochastic {D}ifferential {E}quations: {T}heory and
  {A}pplications} {{S}tochastic {D}ifferential {E}quations: {T}heory and
  {A}pplications}\ (\PrintOrdinal{1}\ \BEd).
\newblock
\APACaddressPublisher{New York, NY}{John Wiley}.
\newblock
\APACrefnote{{ISBN}: 0-471-03359-6}
\PrintBackRefs{\CurrentBib}

\bibitem [\protect \citeauthoryear {%
Berger%
, Ilchmann%
\BCBL {}\ \BBA {} Trenn%
}{%
Berger%
\ \protect \BOthers {.}}{%
{\protect \APACyear {2012}}%
}]{%
BerIT12}
\APACinsertmetastar {%
BerIT12}%
\begin{APACrefauthors}%
Berger, T.%
, Ilchmann, A.%
\BCBL {}\ \BBA {} Trenn, S.%
\end{APACrefauthors}%
\unskip\
\newblock
\APACrefYearMonthDay{2012}{}{}.
\newblock
{\BBOQ}\APACrefatitle {The quasi-{W}eierstra{\ss} form for regular matrix
  pencils} {The quasi-{W}eierstra{\ss} form for regular matrix pencils}.{\BBCQ}
\newblock
\APACjournalVolNumPages{Linear Algebra Appl.}{436}{10}{4052--4069}.
\PrintBackRefs{\CurrentBib}

\bibitem [\protect \citeauthoryear {%
Berger%
\ \BBA {} Reis%
}{%
Berger%
\ \BBA {} Reis%
}{%
{\protect \APACyear {2013}}%
}]{%
BR13}
\APACinsertmetastar {%
BR13}%
\begin{APACrefauthors}%
Berger, T.%
\BCBT {}\ \BBA {} Reis, T.%
\end{APACrefauthors}%
\unskip\
\newblock
\APACrefYearMonthDay{2013}{}{}.
\newblock
{\BBOQ}\APACrefatitle {Controllability of linear differential-algebraic systems
  -- a survey} {Controllability of linear differential-algebraic systems -- a
  survey}.{\BBCQ}
\newblock
\BIn{} A.~Ilchmann\ \BBA {} T.~Reis\ (\BEDS), \APACrefbtitle {Surveys in
  {D}ifferential-{A}lgebraic {E}quations {I}} {Surveys in
  {D}ifferential-{A}lgebraic {E}quations {I}}\ (\BPGS\ 1--61).
\newblock
\APACaddressPublisher{Berlin}{Springer}.
\PrintBackRefs{\CurrentBib}

\bibitem [\protect \citeauthoryear {%
Bijl%
\ \BBA {} Sch{\"o}n%
}{%
Bijl%
\ \BBA {} Sch{\"o}n%
}{%
{\protect \APACyear {2019}}%
}]{%
BS17}
\APACinsertmetastar {%
BS17}%
\begin{APACrefauthors}%
Bijl, H.%
\BCBT {}\ \BBA {} Sch{\"o}n, T\BPBI B.%
\end{APACrefauthors}%
\unskip\
\newblock
\APACrefYearMonthDay{2019}{}{}.
\newblock
{\BBOQ}\APACrefatitle {Optimal controller/observer gains of discounted-cost
  {LQG} systems} {Optimal controller/observer gains of discounted-cost {LQG}
  systems}.{\BBCQ}
\newblock
\APACjournalVolNumPages{Automatica J. IFAC}{101}{}{471--474}.
\PrintBackRefs{\CurrentBib}

\bibitem [\protect \citeauthoryear {%
Bijl%
, van Wingerden%
, Sch{\"o}n%
\BCBL {}\ \BBA {} Verhaegen%
}{%
Bijl%
\ \protect \BOthers {.}}{%
{\protect \APACyear {2016}}%
}]{%
BWSV16}
\APACinsertmetastar {%
BWSV16}%
\begin{APACrefauthors}%
Bijl, H.%
, van Wingerden, J\BHBI W.%
, Sch{\"o}n, T\BPBI B.%
\BCBL {}\ \BBA {} Verhaegen, M.%
\end{APACrefauthors}%
\unskip\
\newblock
\APACrefYearMonthDay{2016}{}{}.
\newblock
{\BBOQ}\APACrefatitle {Mean and variance of the {LQG} cost function} {Mean and
  variance of the {LQG} cost function}.{\BBCQ}
\newblock
\APACjournalVolNumPages{Automatica J. IFAC}{67}{}{216--223}.
\PrintBackRefs{\CurrentBib}

\bibitem [\protect \citeauthoryear {%
Chen%
, Li%
\BCBL {}\ \BBA {} Zhou%
}{%
Chen%
\ \protect \BOthers {.}}{%
{\protect \APACyear {1998}}%
}]{%
CLZ98}
\APACinsertmetastar {%
CLZ98}%
\begin{APACrefauthors}%
Chen, S.%
, Li, X.%
\BCBL {}\ \BBA {} Zhou, X\BPBI Y.%
\end{APACrefauthors}%
\unskip\
\newblock
\APACrefYearMonthDay{1998}{}{}.
\newblock
{\BBOQ}\APACrefatitle {Stochastic linear quadratic regulators with indefinite
  control weight costs} {Stochastic linear quadratic regulators with indefinite
  control weight costs}.{\BBCQ}
\newblock
\APACjournalVolNumPages{SIAM J. Control Optim.}{36}{5}{1685--1702}.
\PrintBackRefs{\CurrentBib}

\bibitem [\protect \citeauthoryear {%
Dai%
}{%
Dai%
}{%
{\protect \APACyear {1989}}%
}]{%
D89lqg}
\APACinsertmetastar {%
D89lqg}%
\begin{APACrefauthors}%
Dai, L.%
\end{APACrefauthors}%
\unskip\
\newblock
\APACrefYearMonthDay{1989}{}{}.
\newblock
{\BBOQ}\APACrefatitle {Filtering and {LQG} problems for discrete-time
  stochastic singular systems} {Filtering and {LQG} problems for discrete-time
  stochastic singular systems}.{\BBCQ}
\newblock
\APACjournalVolNumPages{IEEE Trans. Automat. Control}{34}{10}{1105--1108}.
\PrintBackRefs{\CurrentBib}

\bibitem [\protect \citeauthoryear {%
Damm%
}{%
Damm%
}{%
{\protect \APACyear {2004}}%
}]{%
D04}
\APACinsertmetastar {%
D04}%
\begin{APACrefauthors}%
Damm, T.%
\end{APACrefauthors}%
\unskip\
\newblock
\APACrefYear{2004}.
\newblock
\APACrefbtitle {Rational {M}atrix {E}quations in {S}tochastic {C}ontrol}
  {Rational {M}atrix {E}quations in {S}tochastic {C}ontrol}\
  (\PrintOrdinal{1st}\ \BEd).
\newblock
\APACaddressPublisher{Berlin}{Springer}.
\newblock
\APACrefnote{{ISBN}: 3-540-20516-0}
\PrintBackRefs{\CurrentBib}

\bibitem [\protect \citeauthoryear {%
Duncan%
, Guo%
\BCBL {}\ \BBA {} Pasik-Duncan%
}{%
Duncan%
\ \protect \BOthers {.}}{%
{\protect \APACyear {1999}}%
}]{%
DGP99}
\APACinsertmetastar {%
DGP99}%
\begin{APACrefauthors}%
Duncan, T\BPBI E.%
, Guo, L.%
\BCBL {}\ \BBA {} Pasik-Duncan, B.%
\end{APACrefauthors}%
\unskip\
\newblock
\APACrefYearMonthDay{1999}{}{}.
\newblock
{\BBOQ}\APACrefatitle {Adaptive continuous-time linear quadratic {G}aussian
  control} {Adaptive continuous-time linear quadratic {G}aussian
  control}.{\BBCQ}
\newblock
\APACjournalVolNumPages{IEEE Trans. Automat. Control}{44}{9}{1653--1662}.
\PrintBackRefs{\CurrentBib}

\bibitem [\protect \citeauthoryear {%
Duncan%
\ \BBA {} Pasik-Duncan%
}{%
Duncan%
\ \BBA {} Pasik-Duncan%
}{%
{\protect \APACyear {2013}}%
}]{%
DPD13}
\APACinsertmetastar {%
DPD13}%
\begin{APACrefauthors}%
Duncan, T\BPBI E.%
\BCBT {}\ \BBA {} Pasik-Duncan, B.%
\end{APACrefauthors}%
\unskip\
\newblock
\APACrefYearMonthDay{2013}{}{}.
\newblock
{\BBOQ}\APACrefatitle {Ergodic problems for linear exponential quadratic
  {G}aussian control and linear quadratic stochastic differential games}
  {Ergodic problems for linear exponential quadratic {G}aussian control and
  linear quadratic stochastic differential games}.{\BBCQ}
\newblock
\BIn{} \APACrefbtitle {Proceedings of the 52nd {IEEE} {C}onference on
  {D}ecision and {C}ontrol} {Proceedings of the 52nd {IEEE} {C}onference on
  {D}ecision and {C}ontrol}\ (\BPGS\ 2488--2492).
\newblock
\APACaddressPublisher{Florence, Italy}{}.
\PrintBackRefs{\CurrentBib}

\bibitem [\protect \citeauthoryear {%
Feng%
, Cui%
\BCBL {}\ \BBA {} Hou%
}{%
Feng%
\ \protect \BOthers {.}}{%
{\protect \APACyear {2013}}%
}]{%
FCH13}
\APACinsertmetastar {%
FCH13}%
\begin{APACrefauthors}%
Feng, J\BHBI e.%
, Cui, P.%
\BCBL {}\ \BBA {} Hou, Z.%
\end{APACrefauthors}%
\unskip\
\newblock
\APACrefYearMonthDay{2013}{}{}.
\newblock
{\BBOQ}\APACrefatitle {Singular linear quadratic optimal control for singular
  stochastic discrete-time systems} {Singular linear quadratic optimal control
  for singular stochastic discrete-time systems}.{\BBCQ}
\newblock
\APACjournalVolNumPages{Optimal Control Appl. Methods}{34}{5}{505--516}.
\PrintBackRefs{\CurrentBib}

\bibitem [\protect \citeauthoryear {%
Gao%
\ \BBA {} Shi%
}{%
Gao%
\ \BBA {} Shi%
}{%
{\protect \APACyear {2013}}%
}]{%
GS13}
\APACinsertmetastar {%
GS13}%
\begin{APACrefauthors}%
Gao, Z.%
\BCBT {}\ \BBA {} Shi, X.%
\end{APACrefauthors}%
\unskip\
\newblock
\APACrefYearMonthDay{2013}{}{}.
\newblock
{\BBOQ}\APACrefatitle {Observer-based controller design for stochastic
  descriptor systems with {B}rownian motions} {Observer-based controller design
  for stochastic descriptor systems with {B}rownian motions}.{\BBCQ}
\newblock
\APACjournalVolNumPages{Automatica J. IFAC}{49}{7}{2229--2235}.
\PrintBackRefs{\CurrentBib}

\bibitem [\protect \citeauthoryear {%
Geerts%
}{%
Geerts%
}{%
{\protect \APACyear {1994}}%
{\protect \APACexlab {{\protect \BCnt {1}}}}}]{%
Gee94}
\APACinsertmetastar {%
Gee94}%
\begin{APACrefauthors}%
Geerts, T.%
\end{APACrefauthors}%
\unskip\
\newblock
\APACrefYearMonthDay{1994{\protect \BCnt {1}}}{}{}.
\newblock
{\BBOQ}\APACrefatitle {Linear-Quadratic Control With and Without Stability
  Subject to General Implicit Continuous-Time Systems: Coordinate-Free
  Interpretations of the Optimal Costs in Terms of Dissipation Inequality and
  Linear Matrix Inequality; Existence and Uniqueness of Optimal Controls and
  State Trajectories} {Linear-quadratic control with and without stability
  subject to general implicit continuous-time systems: Coordinate-free
  interpretations of the optimal costs in terms of dissipation inequality and
  linear matrix inequality; existence and uniqueness of optimal controls and
  state trajectories}.{\BBCQ}
\newblock
\APACjournalVolNumPages{Linear Algebra Appl.}{203--204}{}{607--658}.
\PrintBackRefs{\CurrentBib}

\bibitem [\protect \citeauthoryear {%
Geerts%
}{%
Geerts%
}{%
{\protect \APACyear {1994}}%
{\protect \APACexlab {{\protect \BCnt {2}}}}}]{%
Gee93}
\APACinsertmetastar {%
Gee93}%
\begin{APACrefauthors}%
Geerts, T.%
\end{APACrefauthors}%
\unskip\
\newblock
\APACrefYearMonthDay{1994{\protect \BCnt {2}}}{}{}.
\newblock
{\BBOQ}\APACrefatitle {Regularity and singularity in linear-quadratic control
  subject to implicit continuous-time systems} {Regularity and singularity in
  linear-quadratic control subject to implicit continuous-time systems}.{\BBCQ}
\newblock
\APACjournalVolNumPages{Circuits Systems Signal Process.}{13}{1}{19--30}.
\PrintBackRefs{\CurrentBib}

\bibitem [\protect \citeauthoryear {%
Huang%
\ \BBA {} Mao%
}{%
Huang%
\ \BBA {} Mao%
}{%
{\protect \APACyear {2010}}%
}]{%
HM10}
\APACinsertmetastar {%
HM10}%
\begin{APACrefauthors}%
Huang, L.%
\BCBT {}\ \BBA {} Mao, X.%
\end{APACrefauthors}%
\unskip\
\newblock
\APACrefYearMonthDay{2010}{}{}.
\newblock
{\BBOQ}\APACrefatitle {Stability of singular stochastic systems with
  {M}arkovian switching} {Stability of singular stochastic systems with
  {M}arkovian switching}.{\BBCQ}
\newblock
\APACjournalVolNumPages{IEEE Trans. Automat. Control}{56}{2}{424--429}.
\PrintBackRefs{\CurrentBib}

\bibitem [\protect \citeauthoryear {%
Kawamoto%
\ \BBA {} Katayama%
}{%
Kawamoto%
\ \BBA {} Katayama%
}{%
{\protect \APACyear {2002}}%
}]{%
KawK02}
\APACinsertmetastar {%
KawK02}%
\begin{APACrefauthors}%
Kawamoto, A.%
\BCBT {}\ \BBA {} Katayama, T.%
\end{APACrefauthors}%
\unskip\
\newblock
\APACrefYearMonthDay{2002}{}{}.
\newblock
{\BBOQ}\APACrefatitle {The semi-stabilizing solution of generalized algebraic
  {R}iccati equation for descriptor systems} {The semi-stabilizing solution of
  generalized algebraic {R}iccati equation for descriptor systems}.{\BBCQ}
\newblock
\APACjournalVolNumPages{Automatica J. IFAC}{38}{10}{1651--1662}.
\PrintBackRefs{\CurrentBib}

\bibitem [\protect \citeauthoryear {%
Kawamoto%
, Takaba%
\BCBL {}\ \BBA {} Katayama%
}{%
Kawamoto%
\ \protect \BOthers {.}}{%
{\protect \APACyear {1999}}%
}]{%
KawTK99}
\APACinsertmetastar {%
KawTK99}%
\begin{APACrefauthors}%
Kawamoto, A.%
, Takaba, K.%
\BCBL {}\ \BBA {} Katayama, T.%
\end{APACrefauthors}%
\unskip\
\newblock
\APACrefYearMonthDay{1999}{}{}.
\newblock
{\BBOQ}\APACrefatitle {On the generalized algebraic {R}iccati equation for
  continuous-time descriptor systems} {On the generalized algebraic {R}iccati
  equation for continuous-time descriptor systems}.{\BBCQ}
\newblock
\APACjournalVolNumPages{Linear Algebra Appl.}{296}{1--3}{1--14}.
\PrintBackRefs{\CurrentBib}

\bibitem [\protect \citeauthoryear {%
Kunkel%
\ \BBA {} Mehrmann%
}{%
Kunkel%
\ \BBA {} Mehrmann%
}{%
{\protect \APACyear {2008}}%
}]{%
KunM08}
\APACinsertmetastar {%
KunM08}%
\begin{APACrefauthors}%
Kunkel, P.%
\BCBT {}\ \BBA {} Mehrmann, V.%
\end{APACrefauthors}%
\unskip\
\newblock
\APACrefYearMonthDay{2008}{}{}.
\newblock
{\BBOQ}\APACrefatitle {Optimal control for unstructured nonlinear
  differential-algebraic equations of arbitrary index} {Optimal control for
  unstructured nonlinear differential-algebraic equations of arbitrary
  index}.{\BBCQ}
\newblock
\APACjournalVolNumPages{Math. Control Signals Systems}{20}{3}{227--269}.
\PrintBackRefs{\CurrentBib}

\bibitem [\protect \citeauthoryear {%
Kunkel%
\ \BBA {} Mehrmann%
}{%
Kunkel%
\ \BBA {} Mehrmann%
}{%
{\protect \APACyear {2011}}%
}]{%
KunM11}
\APACinsertmetastar {%
KunM11}%
\begin{APACrefauthors}%
Kunkel, P.%
\BCBT {}\ \BBA {} Mehrmann, V.%
\end{APACrefauthors}%
\unskip\
\newblock
\APACrefYearMonthDay{2011}{}{}.
\newblock
{\BBOQ}\APACrefatitle {Formal adjoints of linear {DAE} operators and their role
  in optimal control} {Formal adjoints of linear {DAE} operators and their role
  in optimal control}.{\BBCQ}
\newblock
\APACjournalVolNumPages{Electron. J. Linear Algebra}{22}{}{672--693}.
\PrintBackRefs{\CurrentBib}

\bibitem [\protect \citeauthoryear {%
Kurina%
\ \BBA {} M{\"a}rz%
}{%
Kurina%
\ \BBA {} M{\"a}rz%
}{%
{\protect \APACyear {2004}}%
}]{%
KurM04}
\APACinsertmetastar {%
KurM04}%
\begin{APACrefauthors}%
Kurina, G\BPBI A.%
\BCBT {}\ \BBA {} M{\"a}rz, R.%
\end{APACrefauthors}%
\unskip\
\newblock
\APACrefYearMonthDay{2004}{}{}.
\newblock
{\BBOQ}\APACrefatitle {On linear-quadratic optimal control problems for
  time-varying descriptor systems} {On linear-quadratic optimal control
  problems for time-varying descriptor systems}.{\BBCQ}
\newblock
\APACjournalVolNumPages{SIAM J. Control Optim.}{42}{6}{2062--2077}.
\PrintBackRefs{\CurrentBib}

\bibitem [\protect \citeauthoryear {%
Kwakernaak%
\ \BBA {} Sivan%
}{%
Kwakernaak%
\ \BBA {} Sivan%
}{%
{\protect \APACyear {1972}}%
}]{%
KS72}
\APACinsertmetastar {%
KS72}%
\begin{APACrefauthors}%
Kwakernaak, H.%
\BCBT {}\ \BBA {} Sivan, R.%
\end{APACrefauthors}%
\unskip\
\newblock
\APACrefYear{1972}.
\newblock
\APACrefbtitle {{L}inear {O}ptimal {C}ontrol {S}ystems} {{L}inear {O}ptimal
  {C}ontrol {S}ystems}\ (\PrintOrdinal{1}\ \BEd).
\newblock
\APACaddressPublisher{New York, NY}{Wiley-Interscience}.
\newblock
\APACrefnote{{ISBN}: 0-471-51110-2}
\PrintBackRefs{\CurrentBib}

\bibitem [\protect \citeauthoryear {%
Masubuchi%
}{%
Masubuchi%
}{%
{\protect \APACyear {2006}}%
}]{%
Mas06}
\APACinsertmetastar {%
Mas06}%
\begin{APACrefauthors}%
Masubuchi, I.%
\end{APACrefauthors}%
\unskip\
\newblock
\APACrefYearMonthDay{2006}{}{}.
\newblock
{\BBOQ}\APACrefatitle {Dissipativity inequalities for continuous-time
  descriptor systems with applications to synthesis of control gains}
  {Dissipativity inequalities for continuous-time descriptor systems with
  applications to synthesis of control gains}.{\BBCQ}
\newblock
\APACjournalVolNumPages{Systems Control Lett.}{55}{2}{158--164}.
\PrintBackRefs{\CurrentBib}

\bibitem [\protect \citeauthoryear {%
Mazurov%
\ \BBA {} Pakshin%
}{%
Mazurov%
\ \BBA {} Pakshin%
}{%
{\protect \APACyear {2009}}%
}]{%
MP09}
\APACinsertmetastar {%
MP09}%
\begin{APACrefauthors}%
Mazurov, A.%
\BCBT {}\ \BBA {} Pakshin, P.%
\end{APACrefauthors}%
\unskip\
\newblock
\APACrefYearMonthDay{2009}{}{}.
\newblock
{\BBOQ}\APACrefatitle {Dissipative stochastic differential systems with
  risk-sensitive storage function and control design problems} {Dissipative
  stochastic differential systems with risk-sensitive storage function and
  control design problems}.{\BBCQ}
\newblock
\APACjournalVolNumPages{J. Comp. Syst. Sci. Int.}{48}{5}{705--717}.
\PrintBackRefs{\CurrentBib}

\bibitem [\protect \citeauthoryear {%
Mehrmann%
}{%
Mehrmann%
}{%
{\protect \APACyear {1989}}%
}]{%
Meh89}
\APACinsertmetastar {%
Meh89}%
\begin{APACrefauthors}%
Mehrmann, V.%
\end{APACrefauthors}%
\unskip\
\newblock
\APACrefYearMonthDay{1989}{}{}.
\newblock
{\BBOQ}\APACrefatitle {Existence, uniqueness and stability of solutions to
  singular linear quadratic optimal control problems} {Existence, uniqueness
  and stability of solutions to singular linear quadratic optimal control
  problems}.{\BBCQ}
\newblock
\APACjournalVolNumPages{Linear Algebra Appl.}{121}{}{291--331}.
\PrintBackRefs{\CurrentBib}

\bibitem [\protect \citeauthoryear {%
Mehrmann%
}{%
Mehrmann%
}{%
{\protect \APACyear {1991}}%
}]{%
Meh91}
\APACinsertmetastar {%
Meh91}%
\begin{APACrefauthors}%
Mehrmann, V.%
\end{APACrefauthors}%
\unskip\
\newblock
\APACrefYear{1991}.
\newblock
\APACrefbtitle {The {A}utonomous {L}inear {Q}uadratic {C}ontrol {P}roblem} {The
  {A}utonomous {L}inear {Q}uadratic {C}ontrol {P}roblem}.
\newblock
\APACaddressPublisher{Heidelberg}{Springer-Verlag}.
\newblock
\APACrefnote{{ISBN}: 978-3-540-54170-7}
\PrintBackRefs{\CurrentBib}

\bibitem [\protect \citeauthoryear {%
Rajpurohit%
\ \BBA {} Haddad%
}{%
Rajpurohit%
\ \BBA {} Haddad%
}{%
{\protect \APACyear {2016}}%
}]{%
RH16}
\APACinsertmetastar {%
RH16}%
\begin{APACrefauthors}%
Rajpurohit, T.%
\BCBT {}\ \BBA {} Haddad, W\BPBI M.%
\end{APACrefauthors}%
\unskip\
\newblock
\APACrefYearMonthDay{2016}{}{}.
\newblock
{\BBOQ}\APACrefatitle {Dissipativity theory for nonlinear stochastic dynamical
  systems} {Dissipativity theory for nonlinear stochastic dynamical
  systems}.{\BBCQ}
\newblock
\APACjournalVolNumPages{IEEE Trans. Automat. Control}{62}{4}{1684--1699}.
\PrintBackRefs{\CurrentBib}

\bibitem [\protect \citeauthoryear {%
Rami%
, Moore%
\BCBL {}\ \BBA {} Zhou%
}{%
Rami%
\ \protect \BOthers {.}}{%
{\protect \APACyear {2002}}%
}]{%
RMZ02}
\APACinsertmetastar {%
RMZ02}%
\begin{APACrefauthors}%
Rami, M\BPBI A.%
, Moore, J\BPBI B.%
\BCBL {}\ \BBA {} Zhou, X\BPBI Y.%
\end{APACrefauthors}%
\unskip\
\newblock
\APACrefYearMonthDay{2002}{}{}.
\newblock
{\BBOQ}\APACrefatitle {Indefinite stochastic linear quadratic control and
  generalized differential {R}iccati equation} {Indefinite stochastic linear
  quadratic control and generalized differential {R}iccati equation}.{\BBCQ}
\newblock
\APACjournalVolNumPages{SIAM J. Control Optim.}{40}{4}{1296--1311}.
\PrintBackRefs{\CurrentBib}

\bibitem [\protect \citeauthoryear {%
Reis%
}{%
Reis%
}{%
{\protect \APACyear {2011}}%
}]{%
Rei11}
\APACinsertmetastar {%
Rei11}%
\begin{APACrefauthors}%
Reis, T.%
\end{APACrefauthors}%
\unskip\
\newblock
\APACrefYearMonthDay{2011}{}{}.
\newblock
{\BBOQ}\APACrefatitle {Lur'e equations and even matrix pencils} {Lur'e
  equations and even matrix pencils}.{\BBCQ}
\newblock
\APACjournalVolNumPages{Linear Algebra Appl.}{434}{}{152--173}.
\PrintBackRefs{\CurrentBib}

\bibitem [\protect \citeauthoryear {%
Reis%
, Rendel%
\BCBL {}\ \BBA {} Voigt%
}{%
Reis%
\ \protect \BOthers {.}}{%
{\protect \APACyear {2015}}%
}]{%
RRV15}
\APACinsertmetastar {%
RRV15}%
\begin{APACrefauthors}%
Reis, T.%
, Rendel, O.%
\BCBL {}\ \BBA {} Voigt, M.%
\end{APACrefauthors}%
\unskip\
\newblock
\APACrefYearMonthDay{2015}{}{}.
\newblock
{\BBOQ}\APACrefatitle {The {K}alman--{Y}akubovich--{P}opov inequality for
  differential--algebraic systems} {The {K}alman--{Y}akubovich--{P}opov
  inequality for differential--algebraic systems}.{\BBCQ}
\newblock
\APACjournalVolNumPages{Linear Algebra Appl.}{485}{}{153--193}.
\PrintBackRefs{\CurrentBib}

\bibitem [\protect \citeauthoryear {%
Reis%
\ \BBA {} Voigt%
}{%
Reis%
\ \BBA {} Voigt%
}{%
{\protect \APACyear {2019}}%
}]{%
RV18}
\APACinsertmetastar {%
RV18}%
\begin{APACrefauthors}%
Reis, T.%
\BCBT {}\ \BBA {} Voigt, M.%
\end{APACrefauthors}%
\unskip\
\newblock
\APACrefYearMonthDay{2019}{}{}.
\newblock
{\BBOQ}\APACrefatitle {Linear-quadratic optimal control of
  differential-algebraic systems: The infinite time horizon problem with zero
  terminal state} {Linear-quadratic optimal control of differential-algebraic
  systems: The infinite time horizon problem with zero terminal state}.{\BBCQ}
\newblock
\APACjournalVolNumPages{SIAM J. Control Optim.}{57}{3}{1567--1596}.
\PrintBackRefs{\CurrentBib}

\bibitem [\protect \citeauthoryear {%
Sun%
, Li%
\BCBL {}\ \BBA {} Yong%
}{%
Sun%
\ \protect \BOthers {.}}{%
{\protect \APACyear {2016}}%
}]{%
SLY16}
\APACinsertmetastar {%
SLY16}%
\begin{APACrefauthors}%
Sun, J.%
, Li, X.%
\BCBL {}\ \BBA {} Yong, J.%
\end{APACrefauthors}%
\unskip\
\newblock
\APACrefYearMonthDay{2016}{}{}.
\newblock
{\BBOQ}\APACrefatitle {Open-loop and closed-loop solvabilities for stochastic
  linear quadratic optimal control problems} {Open-loop and closed-loop
  solvabilities for stochastic linear quadratic optimal control
  problems}.{\BBCQ}
\newblock
\APACjournalVolNumPages{SIAM J. Control Optim.}{54}{5}{2274--2308}.
\PrintBackRefs{\CurrentBib}

\bibitem [\protect \citeauthoryear {%
Voigt%
}{%
Voigt%
}{%
{\protect \APACyear {2015}}%
}]{%
V15}
\APACinsertmetastar {%
V15}%
\begin{APACrefauthors}%
Voigt, M.%
\end{APACrefauthors}%
\unskip\
\newblock
\APACrefYear{2015}.
\newblock
\APACrefbtitle {On {L}inear-{Q}uadratic {O}ptimal {C}ontrol and {R}obustness of
  {D}ifferential-{A}lgebraic {S}ystems} {On {L}inear-{Q}uadratic {O}ptimal
  {C}ontrol and {R}obustness of {D}ifferential-{A}lgebraic {S}ystems}.
\newblock
\APACaddressPublisher{}{Logos Verlag Berlin GmbH}.
\newblock
\APACrefnote{{ISBN}: 978-3-8325-4118-7}
\PrintBackRefs{\CurrentBib}

\bibitem [\protect \citeauthoryear {%
Wang%
\ \BBA {} Liu%
}{%
Wang%
\ \BBA {} Liu%
}{%
{\protect \APACyear {2018}}%
}]{%
WL18}
\APACinsertmetastar {%
WL18}%
\begin{APACrefauthors}%
Wang, X.%
\BCBT {}\ \BBA {} Liu, B.%
\end{APACrefauthors}%
\unskip\
\newblock
\APACrefYearMonthDay{2018}{}{}.
\newblock
{\BBOQ}\APACrefatitle {Singular linear quadratic optimal control problem for
  stochastic nonregular descriptor systems} {Singular linear quadratic optimal
  control problem for stochastic nonregular descriptor systems}.{\BBCQ}
\newblock
\APACjournalVolNumPages{Asian J. Control}{20}{5}{1782--1792}.
\PrintBackRefs{\CurrentBib}

\bibitem [\protect \citeauthoryear {%
Willems%
}{%
Willems%
}{%
{\protect \APACyear {1971}}%
}]{%
Wil71}
\APACinsertmetastar {%
Wil71}%
\begin{APACrefauthors}%
Willems, J\BPBI C.%
\end{APACrefauthors}%
\unskip\
\newblock
\APACrefYearMonthDay{1971}{}{}.
\newblock
{\BBOQ}\APACrefatitle {Least squares stationary optimal control and the
  algebraic {R}iccati equation} {Least squares stationary optimal control and
  the algebraic {R}iccati equation}.{\BBCQ}
\newblock
\APACjournalVolNumPages{IEEE Trans. Automat. Control}{AC-16}{6}{621-634}.
\PrintBackRefs{\CurrentBib}

\bibitem [\protect \citeauthoryear {%
Xing%
, Deng%
\BCBL {}\ \BBA {} Qiao%
}{%
Xing%
\ \protect \BOthers {.}}{%
{\protect \APACyear {2018}}%
}]{%
XDQ18}
\APACinsertmetastar {%
XDQ18}%
\begin{APACrefauthors}%
Xing, S.%
, Deng, F.%
\BCBL {}\ \BBA {} Qiao, L.%
\end{APACrefauthors}%
\unskip\
\newblock
\APACrefYearMonthDay{2018}{}{}.
\newblock
{\BBOQ}\APACrefatitle {Dissipative Analysis and Control for Nonlinear
  Stochastic Singular Systems} {Dissipative analysis and control for nonlinear
  stochastic singular systems}.{\BBCQ}
\newblock
\APACjournalVolNumPages{IEEE Access}{6}{}{43070--43078}.
\PrintBackRefs{\CurrentBib}

\bibitem [\protect \citeauthoryear {%
Xing%
\ \BBA {} Zhang%
}{%
Xing%
\ \BBA {} Zhang%
}{%
{\protect \APACyear {2016}}%
}]{%
XZ16}
\APACinsertmetastar {%
XZ16}%
\begin{APACrefauthors}%
Xing, S.%
\BCBT {}\ \BBA {} Zhang, Q.%
\end{APACrefauthors}%
\unskip\
\newblock
\APACrefYearMonthDay{2016}{}{}.
\newblock
{\BBOQ}\APACrefatitle {Stability and exact observability of discrete stochastic
  singular systems based on generalised {L}yapunov equations} {Stability and
  exact observability of discrete stochastic singular systems based on
  generalised {L}yapunov equations}.{\BBCQ}
\newblock
\APACjournalVolNumPages{IET Control Theory Appl.}{10}{9}{971--980}.
\PrintBackRefs{\CurrentBib}

\bibitem [\protect \citeauthoryear {%
Yong%
\ \BBA {} Zhou%
}{%
Yong%
\ \BBA {} Zhou%
}{%
{\protect \APACyear {1999}}%
}]{%
YZ99}
\APACinsertmetastar {%
YZ99}%
\begin{APACrefauthors}%
Yong, J.%
\BCBT {}\ \BBA {} Zhou, X\BPBI Y.%
\end{APACrefauthors}%
\unskip\
\newblock
\APACrefYear{1999}.
\newblock
\APACrefbtitle {Stochastic {C}ontrols: {H}amiltonian {S}ystems and {HJB}
  {E}quations} {Stochastic {C}ontrols: {H}amiltonian {S}ystems and {HJB}
  {E}quations}\ (\PrintOrdinal{1st}\ \BEd).
\newblock
\APACaddressPublisher{New York, NY}{Springer}.
\newblock
\APACrefnote{{ISBN} 0-387-98723-1}
\PrintBackRefs{\CurrentBib}

\bibitem [\protect \citeauthoryear {%
G.~Zhang%
\ \BBA {} James%
}{%
G.~Zhang%
\ \BBA {} James%
}{%
{\protect \APACyear {2010}}%
}]{%
ZJ10}
\APACinsertmetastar {%
ZJ10}%
\begin{APACrefauthors}%
Zhang, G.%
\BCBT {}\ \BBA {} James, M\BPBI R.%
\end{APACrefauthors}%
\unskip\
\newblock
\APACrefYearMonthDay{2010}{}{}.
\newblock
{\BBOQ}\APACrefatitle {Direct and indirect couplings in coherent feedback
  control of linear quantum systems} {Direct and indirect couplings in coherent
  feedback control of linear quantum systems}.{\BBCQ}
\newblock
\APACjournalVolNumPages{IEEE Trans. Automat. Control}{56}{7}{1535--1550}.
\PrintBackRefs{\CurrentBib}

\bibitem [\protect \citeauthoryear {%
Q.~Zhang%
\ \BBA {} Xing%
}{%
Q.~Zhang%
\ \BBA {} Xing%
}{%
{\protect \APACyear {2014}}%
}]{%
ZX14}
\APACinsertmetastar {%
ZX14}%
\begin{APACrefauthors}%
Zhang, Q.%
\BCBT {}\ \BBA {} Xing, S.%
\end{APACrefauthors}%
\unskip\
\newblock
\APACrefYearMonthDay{2014}{}{}.
\newblock
{\BBOQ}\APACrefatitle {Stability analysis and optimal control of stochastic
  singular systems} {Stability analysis and optimal control of stochastic
  singular systems}.{\BBCQ}
\newblock
\APACjournalVolNumPages{Optim. Lett.}{8}{6}{1905--1920}.
\PrintBackRefs{\CurrentBib}

\bibitem [\protect \citeauthoryear {%
W.~Zhang%
, Zhao%
\BCBL {}\ \BBA {} Sheng%
}{%
W.~Zhang%
\ \protect \BOthers {.}}{%
{\protect \APACyear {2015}}%
}]{%
ZZS15}
\APACinsertmetastar {%
ZZS15}%
\begin{APACrefauthors}%
Zhang, W.%
, Zhao, Y.%
\BCBL {}\ \BBA {} Sheng, L.%
\end{APACrefauthors}%
\unskip\
\newblock
\APACrefYearMonthDay{2015}{}{}.
\newblock
{\BBOQ}\APACrefatitle {Some remarks on stability of stochastic singular systems
  with state-dependent noise} {Some remarks on stability of stochastic singular
  systems with state-dependent noise}.{\BBCQ}
\newblock
\APACjournalVolNumPages{Automatica J. IFAC}{51}{}{273--277}.
\PrintBackRefs{\CurrentBib}

\bibitem [\protect \citeauthoryear {%
Zhou%
, Zhang%
, Li%
, Men%
\BCBL {}\ \BBA {} Ren%
}{%
Zhou%
\ \protect \BOthers {.}}{%
{\protect \APACyear {2016}}%
}]{%
ZZLMR16}
\APACinsertmetastar {%
ZZLMR16}%
\begin{APACrefauthors}%
Zhou, J.%
, Zhang, Q.%
, Li, J.%
, Men, B.%
\BCBL {}\ \BBA {} Ren, J.%
\end{APACrefauthors}%
\unskip\
\newblock
\APACrefYearMonthDay{2016}{}{}.
\newblock
{\BBOQ}\APACrefatitle {Dissipative control for a class of nonlinear descriptor
  systems} {Dissipative control for a class of nonlinear descriptor
  systems}.{\BBCQ}
\newblock
\APACjournalVolNumPages{Internat. J. Systems Sci.}{47}{5}{1110--1120}.
\PrintBackRefs{\CurrentBib}

\end{thebibliography}
	
\end{document}